\documentclass[twocolumn,secnumarabic,nobibnotes,aps,pre,reprint,amsmath]{revtex4-2}

\usepackage[pdftex]{graphicx}
\usepackage{amssymb}
\usepackage{amsthm}
\usepackage{array}
\usepackage{bm}
\usepackage{color}
\usepackage{etex}
\usepackage{hyperref}
\usepackage{mathtools}
\usepackage{stackrel}

\usepackage{cleveref}

\DeclareMathOperator{\Cov}{Cov}
\DeclareMathOperator{\Deriv}{D}
\DeclareMathOperator{\Dom}{dom}
\DeclareMathOperator{\Entropy}{\mathcal H}
\DeclareMathOperator{\Expectation}{\mathbb E}
\DeclareMathOperator{\Hessian}{Hess}
\DeclareMathOperator{\Supp}{Supp}
\DeclareMathOperator{\Var}{Var}
\DeclareMathOperator{\diag}{diag}

\newcommand{\KL}[2]{\operatorname{D}\left(#1\,\Vert#2\right)}
\newcommand{\R}{\mathbb{R}}

\newcommand{\affineof}[1]{\operatorname{Affine}\left(#1\right)}
\newcommand{\be}{\bm e}
\newcommand{\btheta}{\bm{\theta}}
\newcommand{\cosof}[1]{\cos\left(#1\right)}
\newcommand{\cossqof}[1]{\cosof{#1}}
\newcommand{\covat}[3]{\Cov_{#1}\left(#2,#3\right)}
\newcommand{\dderivby}[1]{\frac{d^2}{d{#1}^2}}
\newcommand{\derivby}[1]{\frac{d}{d#1}}
\newcommand{\diagof}[1]{\diag\left(#1\right)}
\newcommand{\domof}[1]{\Dom\left(#1\right)}
\newcommand{\eacc}{\prescript{e}{}\Deriv^2}
\newcommand{\ehessianat}[2]{\prescript{e}{} {\Hessian}_{#1} {#2}}
\newcommand{\entropyof}[1]{\Entropy\left(#1\right)}
\newcommand{\epiof}[1]{\operatorname{epi}\left(#1\right)}
\newcommand{\etransport}[2]{\prescript{\text{e}}{} {\mathbb U} _ {#1} ^ {#2}}
\newcommand{\euler}{\mathrm{e}}
\newcommand{\expectat}[2]{\Expectation_{#1}\left[#2\right]} 

\newcommand{\expof}[1]{\exp\left(#1\right)}
\newcommand{\grad}{\operatorname{grad}}
\newcommand{\logof}[1]{\log\left(#1\right)}
\newcommand{\macc}{\prescript{m}{}\Deriv^2}
\newcommand{\mhessianat}[2]{\prescript{m}{} {\Hessian}_{#1} {#2}}
\newcommand{\mtransport}[2]{\prescript{\text{m}}{} {\mathbb U} _ {#1} ^ {#2}}
\newcommand{\nonnegreals}{\R_{\ge 0}}
\newcommand{\normat}[2]{\left\Vert#2\right\Vert_{#1}}
\newcommand{\normof}[1]{\left\Vert#1\right\Vert}
\newcommand{\one}{\bm 1}
\newcommand{\osimplexof}[1]{\Delta^\circ\left(#1\right)}
\newcommand{\partiald}[2]{\frac{\partial}{\partial #1} #2}
\newcommand{\pderivby}[1]{\frac {\partial} {\partial #1}}
\newcommand{\saffineof}[1]{S\operatorname{Affine}\left(#1\right)}
\newcommand{\scalarat}[3]{\left\langle#2,#3\right\rangle_{#1}}
\newcommand{\scalarof}[2]{\left\langle#1,#2\right\rangle}
\newcommand{\setof}[2]{\left\{#1 \, \middle| \, #2 \right\}}
\newcommand{\set}[1]{\left\{#1\right\}}
\newcommand{\simplexof}[1]{\Delta\left(#1\right)}
\newcommand{\sinof}[1]{\sin\left(#1\right)}
\newcommand{\sinsqof}[1]{\sinof{#1}}
\newcommand{\smallo}{\operatorname{o}}
\newcommand{\suppof}[1]{\Supp\left(#1\right)}
\newcommand{\transport}[2]{\prescript{0}{} {\mathbb U} _ {#1} ^ {#2}}
\newcommand{\varat}[2]{\Var_{#1}\left(#2\right)}

\theoremstyle{plain}
\newtheorem{theorem}{Theorem}
\newtheorem{proposition}{Proposition}
\theoremstyle{definition}
\newtheorem{definition}{Definition}

\theoremstyle{remark}

\begin{document}

\title{Information Geometry of the Probability Simplex: A Short Course}

\author{Giovanni Pistone}
\email[]{giovanni.pistone@carloalberto.org}
\homepage[]{https://www.giannidiorestino.it/}

\thanks{Lectures notes after the 6th Ph.D. School/Conference on
Mathematical Modeling of Complex Systems
Universit\`a ``G. d’Annunzio'', Pescara (IT), July 3--11, 2019. GP is supported by de Castro Statistics and Collegio Carlo Alberto and is a member of INdAM--GNAMPA}

\affiliation{de Castro Statistics, Collegio Carlo Alberto, Piazza Vincenzo Arbarello 8, 10122 Torino, Italy}


\begin{abstract}
This set of notes is intended for a short course aiming to provide an (almost) self-contained and (almost) elementary introduction to the topic of Information Geometry (IG) of the probability simplex. Such a course can be considered an introduction to the original monograph by \citet{amari|nagaoka:2000}, and to the recent monographs by \citet{amari:2016} and by \citet*{Ay|Jost|Le|Schwachhofer:2017IGbook}. The focus is on a non-parametric approach, that is, I consider the geometry of the full probability simplex and compare the IG formalism with what is classically done in Statistical Physics.  
\end{abstract}


\maketitle

\section{Introduction}
\label{sec:introduction}

The key Amari's contribution to our subject has been convincingly showing that Fisher-Rao Riemannian structure on the probability simplex is just one among the geometric structures of interest. In fact, Amari describes the geometry of the probability simplex as an affine space with a natural system of dually flat connections. My aim here is to present Amari's ideas while avoiding the use of parametric differential geometry, my point being that such presentation better reveals the substantial connection with standard arguments in Boltzmann-Gibbs theory as introduced, for example, in \citet[Ch.~I-III]{landau|lifshits:1980}. A second positive effect of the non-parametric approach is to provide a better preparation for interesting generalization, namely, infinite sample space, deformed logarithmic representation, Wasserstein geometry. In the text, I am freely using material from a number of papers that followed \citet{pistone|sempi:95}.

I will not consider here any specific application. In fact, the presentation is limited to the consideration of the basic formalism. If Chance is to be accepted as a real object, as are Space, Time, Space-Time, \dots, then something like IG should be the mathematics of Chance, in the same sense Cartesian Geometry is a mathematics of classical Space.

A statistical model is a parametrised set of probability functions. The point of view of Information geometry (IG) is that a statistical model must be viewed as a sub-manifold of a manifold on all probability functions. This statement requires a number of qualification to be technically feasible. Let us start by considering a few basic examples. 

\paragraph*{1)} On the sample space of two binary trial, $\Omega = \set{0,1}^2$, the set of all possible probability functions is the probability simplex $\Delta(\Omega)$. It is a convex set whose dimension is 3, conveniently represented as 2-way table with elements $p(x,y) \ge 0$, $x,y = 0,1$, $\sum p(x,y) = 1$.

\paragraph*{2)} The model of two independent identically distributed binary trials is a 1-parameter model that can be seen as a curve in the probability simplex of \emph{1)}. One possible parametrization is
\begin{equation*}
[0,1] \ni \theta \mapsto p(x,y;\theta) = ((1-\theta)x  + \theta x)  ((1-\theta)y  + \theta y) \ .
\end{equation*}
Another one is $\theta \mapsto (1-\theta)^{1 - x - y} \theta^{x+y}$. 

\paragraph*{3)} The model of two independent binary trial has 2 independent parameters and can be seen as a surface in the probability simplex. The quadratic homogeneous equation $p(0,0)p(1,1) = p(0,1)p(1,0)$ defines the model as a semi-algebraic surface. It is a ruled surface that can be parametrized on the unit square by
\begin{equation*}
[0,1] \ni (\theta_1,\theta_2) \mapsto ((1-\theta_1)^{1-x}\theta_1^x)((1-\theta_2)^{1-y}\theta_2^y) \ .
\end{equation*}

\paragraph*{4)} The set of all probability functions of the interior of the simplex of \emph{1)} with a given entropy, $\entropyof{p} = - \sum_{x,y} p(x,y) \log p(x,y) = \text{const}$, is a surface of dimension 2.

IG provides the tools for discussing in a geometric language a number of interesting problems about the examples above. For examples, I will define at each point of the open simplex an inner product such that the trajectories that are orthogonal to the surfaces of equal entropy are Gibbs models. 

The language of IG is the language of differential geometry. All the IG monograph quoted above contain a short introduction to differential geometry. Non-parametric presentations of differential geometry can be found in \citet{lang:1995} and \citet{klingenberg:1995}. In these approaches, one the model space (coordinates space) can be any Banach space and different charts of the atlas are not required to have the same image space.

\section{Calculus on the simplex}
\label{sec:recap-convexity}

Convex analysis is a relevant topic in IG. Standard references are the monographs \citet{rockafellar:1970} and \citet{barvinok:2002}. Find below a short review of what is needed for IG.

A subset $H$ of a vector space $V$ is an \emph{affine space} if $TH = \setof{x-y}{x,y \in H}$ is a sub-vector space of $V$. $TH$ is called the vector subspace parallel to $H$ (or tangent to $H$).

Our main example is $V = \R^n$ and $H = \setof{x \in \R^n}{\one^tx = 1}$. If $x,y \in H$, then $\one^t(x-y) = 0$. Conversely, if $\one^t z = 0$, then $z = (z + \be_1) - \be_1$, hence $TH = \setof{z \in \R^n}{\one^t z = 0}$.

The dimension of the affine space $H$ is the dimension of its parallel vector subspace. Given $x_0,\dots,x_n \in V$, the set of all vectors of the form $x_0 + \sum_{j=1}^n \lambda_j x_j$, $\lambda_j \in \R$, is the affine space generated by the given vectors. An affine space of dimension $(n-1)$ in $\R^n$ is an \emph{hyper-plane}.

A subset $C$ of the vector space $V$ is \emph{convex} if for all $x,y \in C$ the segment $(1-\lambda)x + \lambda y$, $\lambda \in [0,1]$ is in $C$. The intersection of two convex sets is clearly convex. Given $x_0,\dots,x_n \in V$ the set of all $\lambda_0 x_0 + \cdots +\lambda_n x_n$ with $\lambda_0 + \cdots + \lambda_n = 1$ is the convex set generated by the given vectors. Such a set is called a \emph{polytope} (or convex polytope). Notice that $\sum_{j=0}^n \lambda_j x_j = (1 - \sum_{j=1}^n \lambda_j) x_0 + \sum_{j=1}^n \lambda_j x_j = x_0 + \sum_{j=1}^n \lambda_j(x_j-x_0)$ that is, the polytope is a part of the affine space generated.

A notable example of convex set is the \emph{half-space} of $v \in V$ such that $\scalarof c v \leq b$ with $c \in V$ and $b \in \R$. A finite intersection of half-spaces is a convex set called a \emph{polyhedron}.

The vectors $x_0,\dots,x_m$ are \emph{affinely independent} if the vectors $x_1-x_0,\dots,x_m-x_0$ are linearly independent. They form a vector basis of the sub-space parallel to the generated polytope which in this case is called a \emph{simplex}. Two simplexes of the same dimension can be mapped one onto the other by an affine transformation that map their respective generators (the vertexes).    

For example, the probability simplex $\Delta(\set{1,2,3})$ and its graphical representation as an equilateral triangle are well known in statistics.

\paragraph*{Example} Let us define more formally the example already used, the probability simplex on $\set{0,1}^2$. Let $V$ be the vector space of real functions on $\set{0,1}^2$, $\R^{\set{0,1}^2} \simeq \R^4$. The four functions $\delta_{ij} = \delta_i \otimes \delta_j$, $i,j=0,1$, are linearly independent, in particular, affinely independent. The convex set generated is the probability simplex $ \simplexof{\set{0,1}^2} =$
  \begin{equation*}
\setof{\sum_{i,j=0,1} p(i,j) \delta_{i,j}}{p(i,j) \geq 0, \sum_{i,j=0,1} p(i,j) = 1} \ .    
\end{equation*}

\paragraph*{Exercise}
Any other set of 4 affinely independent vectors can be used to represent the same probability simplex. For example, the 4 vertexes in $\R^3$ of the tetrahedron are affinely independent,
\begin{equation*}\small
  \begin{array}{cccccc}
    ij & \theta & \phi & x & y & z \\
    \hline
    00 & 0 & 0 & \sinof{0} \cosof{0} & \sinof{0} \cosof{0} & \cosof{0}\\
    01 & \frac23\pi & 0 & \sinof{\frac23\pi} \cosof{0} & \sinof{\frac23\pi} \cosof{0} & \cosof{\frac23\pi}\\
    10 & \frac23\pi & \frac23\pi & \sinof{\frac23\pi} \cosof{\frac23\pi} & \sinof{\frac23\pi} \cosof{\frac23\pi} & \cosof{\frac23\pi}\\
    11 & \frac23\pi & \frac43\pi & \sinof{\frac23\pi} \cosof{\frac43\pi} & \sinof{\frac23\pi} \cosof{\frac43\pi} & \cosof{\frac23\pi}
  \end{array}
\end{equation*}

The possibly most common representation uses the 0 vector together with the vectors of the standard basis, that is the set $\set{p(0,1)\delta_{01}+p(1,0)\delta_{10}+p(11)\delta_{11}}$ with conditions $p(i,j) \geq 0$,  $p(0,1)+p(1,0) + p(1,1) \leq 1$ \ .

\bigskip
Let us recall two basic results about convex sets.

\begin{theorem}
Let $K$ be a convex set of the finite dimensional vector space $V$. Assume $K$ is closed, $K = \overline K$, and its interior is not empty, $K^\circ \neq \emptyset$. Let $x$ be a point of the boundary, $x \in \partial K = \overline K - K^\circ$. The there exist $b \in \R$ and $A \in L(V)$, such that the affine function $h = A + b$ \emph{supports $K$ at $x$}, namely $h(x) = 0$ and $h(y) \geq 0$ if $y \in K$.
\end{theorem}

\begin{proof}
  \cite[\S II.1-2]{barvinok:2002}
\end{proof}

\begin{theorem}
Every polytope is a polyhedron and every bounded polyhedron is a polytope.  
\end{theorem}

\begin{proof}
  \cite[\S II.3]{barvinok:2002}
\end{proof}

At this point, I recap the basic notations of the affine geometry of the probability simplex. Let $\lambda$ be a probability function on $\Omega$. As $\lambda \in \R^\Omega$, one can write $\lambda = \sum_{x \in \Omega} \lambda(x) \delta_x$, so that the set $\Delta(\Omega)$ is the convex set generated by the probability functions associated to the Dirac probability measures. Let us code $\Omega$ as $\set{1,\dots,N}$ and write $\lambda= \sum_{j=1}^n \lambda_j e_j$. The vectors $e_j - e_m$, $j=1,\dots,N-1$ are linearly independent so that $\Delta(\Omega)$ is a special simplex which is called the \emph{probability simplex}. The parallel vector space is the vector space of the vectors of the form $\sum_{j=1}^n \alpha_j (e_j-e_1)$ that is of the form $\sum_{j=1}^n \alpha_j e_j$ with $\sum_{j=1}^n \alpha_j =0$. These are the vectors which are orthogonal to the constant vectors.

The set of probability functions with support $\Omega_1 \subset \Omega$ form a simplex of dimension $\# \Omega_1 -1$. If $\#\Omega_1 = n-1$ this sub-simplex is a \emph{face} of $\Delta(\Omega)$.

There is another simplex that represents the probability simplex $\Delta(\Omega)$ namely, the \emph{solid probability simplex}. In fact, one can represent a probability function by its $n-1$ values $\lambda_j,\dots,\lambda_{n-1}$ which form a vector in $\R^{n-1}$ satisfying the conditions $\lambda_j \geq 0$ and $\sum_{j=1}^{n-1} \lambda_j \leq 1$. The vectors $e_1,\dots,e_{n-1},0 \in \R^{n-1}$ are affinely independent and generate a simplex of dimension $n-1$ as $\sum_{J=1}^{n-1} \lambda_j e_j + \lambda_n 0$. The mapping between the two representations is given by $\R^n \ni e_j \mapsto e_j \in \R^{n-1}$ for $j=1,\dots,n-1$ and $\R^n \ni e_n \mapsto 0 \in \R^{n-1}$.

\bigskip
Let us turn to the calculus on the simplex. Let $f \colon \mathcal O \to \R^n$, where $\mathcal O$ is an open sub-set of $\R^m$. The function is differentiable at $\bar x \in \mathcal O$ if there exists a linear mapping $df(\bar x) \in L(\R^m,\R^n)$ such that
\begin{equation*}
f(\bar x + h) - f(\bar x) - df(\bar x)[h] = \smallo(h) \ .
\end{equation*}
The matrix representing the linear operator $df(\bar x)$ is called the Jacobian matrix of $f$, $Jf(\bar x)$, whose elements are the partial derivatives
\begin{equation*}
  Jf(\bar x) =
  \begin{bmatrix}
    \frac{\partial}{\partial x_j} f_i(x_1,\dots,x_n)
  \end{bmatrix}_{i=1,\dots,n; j=1,\dots m} \ .
\end{equation*}

The derivative of the composite function $f \circ g$ at $x$ is $df\circ g(x) = df(g(x)) \circ dg(x)$.

\paragraph*{Example} Here is a fundamental remark. \emph{Let $I \ni \theta \mapsto \lambda(\theta)$ be a curve which stays in the probability simplex $\Delta(\Omega)$ and which is differentiable in $\R^\Omega$.} The derivative
\begin{equation*}
\lambda'(\theta) = \lim_{h \to 0} h^{-1} (\lambda(\theta+h) - \lambda(\theta)
\end{equation*}
belongs to the subspace parallel to the simplex. If $\lambda(\bar \omega;\bar \theta) = 0$, then the real differentiable function $\theta \mapsto \lambda(\bar\omega,\theta)$ has a minimum at $\theta=\bar\theta$, so that $\lambda'(\bar\omega,\bar\theta)=0$ and $\lambda'(\bar\theta)$ belong to the space parallel to the face of the simplex characterised by $\lambda(\bar\omega) = 0$. In the language of measure theory, \emph{$\lambda'(t)$ is absolutely continuous with respect to $\lambda(t)$, that is, there exists a curve $t \mapsto s(t)$ such that $\lambda'(x;t) = s(x;t) \lambda(x;t)$ for all $x$ and $t$.}  Notice that, if $\lambda(x;t)$ stays positive in some time interval, then one can take $s(x;t) = \derivby t \log \lambda(x;t)$ on that interval.

\paragraph*{Exercise.} The entropy $\entropyof \lambda = -\sum_{\omega} \lambda(\omega) \log \lambda(\omega)$ is defined on the convex set $\Delta^\circ(\Omega)$, $\# \Omega = N$, of strictly positive probability functions. As $\phi(x) = - x\log x$, $x > 0$,  is concave,
  \begin{equation*}
    \frac1{N} \entropyof\lambda) = \frac 1 {N} \sum_{\omega\in\Omega} \phi(\lambda(\omega)) \leq \phi\left(\frac 1 {N} \sum_{\omega\in\Omega}\right) = \phi\left(\frac1{N}\right) \ , 
  \end{equation*}
and the uniform probability function is a maximum of the entropy. Let us show that this maximum is unique. Assume there is a $\bar \lambda$ which is a maximum for the entropy and let $\theta \mapsto \lambda(\theta)$ be a differentiable curve in $\Delta^0(\Omega)$ such that $\lambda(0) = \bar\lambda$. Let us compute the derivative
  \begin{multline*}
   \left. \derivby \theta H(\lambda(\theta))\right|_{\lambda=0} = - \left. \sum_{\omega \in \Omega} (\log \lambda(\omega;\theta)+1) \lambda'(\omega;\theta) \right|_{\theta=0} = \\ - \sum_{\omega \in \Omega} (\log \bar \lambda(\omega) + 1) \lambda'(\omega;0) = 0 \ .
 \end{multline*}

As $\bar\lambda$ is in the $\Delta^0(\Omega)$, for each $v$ in the space parallel to the simplex one can consider the curve $\theta \mapsto \bar\lambda + \theta v$ whose derivative at $\theta=0$ is $v$. It follows that for each $v$ one has
\begin{equation*}
  \sum_{\omega \in \Omega} (\log \bar \lambda(\omega) + 1) v(\omega) = 0
\end{equation*}
hence, $\log \bar \lambda$ is constant that is, $\bar\lambda$ is constant $\bar\lambda(\omega) = 1/N$. 

\medskip

Let us move now to the discussion of convex functions. If a convex set $A \in \R^m$ is open, then every straight line intersects $C$ in an open interval or an empty interval. For example, the subset of the solid probability simplex consisting of strictly positive probability functions is an open convex set. The closure $\overline A$ of an open convex set $A$ is a convex set. The difference $\overline A \setminus A$ is the boundary of the convex set. Let $x$ be a point of the boundary. A unit vector $u$ applied at $x$ enters $A$ if there is a $y \in A$ such that $u = (y-x)/\normof{y-x}$. The set of all entering vectors cannot contain two antipodal elements so that there is a unit vector $w$ such that $\scalarof w u < 0$ for all entering unit vector.

\begin{theorem}[Isolation Theorem]
  Let $A$ be an open convex set in $\R^m$ and let $x$ be in the border of $A$. There exists a unit vector $w$ such that  $\scalarof w {y-x} < 0$ for all $y \in A$ that is, the half-space contains the convex set
\end{theorem}
\begin{proof}See a full proof in \citet[p 45-46]{barvinok:2002}.
\end{proof}

A function $\phi$ defined on $\R^n$ with values in $\overline\R = \R \cup \set{+\infty}$ is convex if the \emph{epigraph} $\epiof \phi = \setof {(x,t)}{ x \in \domof \phi, t \in \R, \phi(x) \leq t}$ is a convex subset of $\R^{n+1}$. Define $\domof \phi$ to be the set where $\phi$ takes finite values. \emph{If $\phi$ is convex, then $\domof \phi$ is a convex subset of $\R^n$.} If $x_1,x_2 \in \domof \phi$, then there exist $(x_1,t_1), (x_2,t_2) \in \epiof \phi$ and for all $\lambda \in [0,1]$ it holds $((1-\lambda)x_1+\lambda x_2,(1-\lambda)t_1+\lambda t_2) \in \epiof \phi$. In particular, $\phi((1-\lambda)x_1+\lambda x_2) < + \infty$. \emph{If $\phi$ is  convex, then $(1-\lambda)\phi(x_1) + \lambda \phi(x_2) \leq \phi((1-\lambda)x_1+\lambda x_2)$ for all $x_1,x_2 \in \R^n$ and $\lambda \in [0,1]$.} If any of $x_1$, $x_2$ in nor in $\domof \phi$ the inequality is trivially satisfied. Otherwise, it is the same computation as above. Conversely, if $\phi \colon \domof \phi \to \R$ and $(1-\lambda)\phi(x_1) + \lambda \phi(x_2) \leq \phi((1-\lambda)x_1+\lambda x_2)$ for all $x_1,x_2 \in \domof \phi$ and $\lambda \in [0,1]$, then the function extended with value $+\infty$ outside the domain is convex.

Let $\phi$ be convex, and define the strict epigraph be open convex set $\setof {(x,t)}{ x \in \domof \phi, t \in \R, \phi(x) < t}$. Assume that at a point $(x,\phi(x))$ the entering unit vectors are not all horizontal. Then the Isolation Theorem implies that there exist at least a \emph{supporting hyper-plane}. In such a case, $\phi$ on all such points \emph{$\phi$ is the point-wise maximum of the supporting affine functions}. In the differentiable case, the tangent plane is the unique supporting hyperplane. If $\phi \in C^2(\mathcal O)$ then the Hessian matrix is non-negative definite.

The following result is important in the theory of exponential families. \emph{Let $\phi$ be convex and let $\phi$ be differentiable on an open $\mathcal O$. Then $\nabla \phi \colon \mathcal O \to \R^n$ is \emph{monotone} i.e., $\scalarof {\nabla \phi(x) - \nabla \phi(y)}{x-y} \geq 0$ for $x,y \in \mathcal O$.} The basic inequality can be rewritten as
\begin{equation*}
  \lambda^{-1}\left(\phi(x + \lambda(y-x))-\phi(x)\right) \leq \phi(y) - \phi(x) \ . 
\end{equation*}
If $\lambda \to 0$, then $\scalarof{\nabla \phi(x)}{y-x} \leq \phi(y) - \phi(x)$. By adding the same inequality with $x$ and $y$ exchanged, one obtains the monotonicity.

Conversely, \emph{if $\phi$ is differentiable and monotone on an open set $\mathcal O$, then $\phi$ is convex on $\mathcal O$.} Write $z = (1-\lambda)x + \lambda y$ and assume $0 < \lambda < 1$ because otherwise there is nothing to prove. Observe that
\begin{widetext}
\begin{multline*}
  \phi(z) - \phi(x) = \int_0^1 \scalarof {\nabla\phi(x + t (z-x))}{z-x} \ dt = \\
  \int_0^1 \scalarof {\nabla\phi(x + t (z-x)) - \nabla\phi(z)}{z-x} \ dt + \scalarof{\nabla\phi(z)}{z-x} \leq \scalarof{\nabla\phi(z)}{z-x} = \lambda \scalarof{\nabla\phi(z)}{y-x} \ .
\end{multline*}
In fact, $z-x$ and $(x + t(x-z)) - z$ are proportional with factor $-(1-t) \leq 0$. In a similar way,
\begin{multline*}
  \phi(y) - \phi(z) = \int_0^1 \scalarof {\nabla\phi(z + t  (y-z))}{y-z} \ dt = \\ \int_0^1 \scalarof {\nabla\phi(z + t (y-z))-\nabla\phi(z)}{y-z} \ dt + \scalarof{\nabla\phi(z)}{y-z} \geq \scalarof{\nabla\phi(z)}{y-z} = (1-\lambda) \scalarof{\nabla\phi(z)}{y-x}\ ,
\end{multline*}
as $y-z$ and $(z+t(y-z)) - z$ are proportional with a factor $t \geq 0$. Rearrange the two inequalities as
\begin{equation*}
  \phi((1-\lambda)x+\lambda y) \leq \phi(x) + \lambda \scalarof{\nabla\phi(z)}{y-x} \text{and} 
  \phi((1-\lambda)x+\lambda y) \leq \phi(y) + (1-\lambda) \scalarof{\nabla\phi(z)}{y-x} 
\end{equation*}
and take the convex combination to conclude the proof. This proof is taken from \citet[p. 26]{rockafellar:1970}.
\end{widetext}

\section{The open simplex}
\label{sec:1a}

Let $\Omega$ be given a finite set with $N = \#\Omega$ points, the \emph{sample space}. Denote by $\simplexof{\Omega}$ the set of the \emph{probability functions} $p \colon \Omega \to \nonnegreals$, $\sum_{x\in\Omega} p(x) = 1$. It is a $(N-1)$-simplex of $\R^\Omega$ that is, an $(N-1)$-dimensional polytope which is the convex hull of its $N$ vertexes $\delta_x$, $x \in \Omega$. It is a closed and convex subset of the affine space $\affineof{\Omega} = \setof{q \in \R^\Omega}{\sum_{x\in\Omega} q(x) = 1}$, the space of \emph{signed probability functions}. It has a non empty relative topological interior $\osimplexof{\Omega}$, which is the set of the strictly positive probability functions,
\begin{equation*}
 \osimplexof{\Omega} = \setof{p \in \R^{\Omega}}{\sum_{x\in\Omega} p(x) = 1, p(x) > 0}.
\end{equation*}
The border of the simplex $\simplexof{\Omega}$ is the union of all its faces as a convex set. Recall that  a face of maximal dimension $(n-1)$ is called \emph{facet}. Each face is itself a simplex. An \emph{edge} is a face of dimension 1. The focus will be now on the geometry of the open simplex $\osimplexof{\Omega}$.

\bigskip
Recall that our aim here is to provide a presentation of Information Geometry in the sense of the monographs by \citet{amari|nagaoka:2000}, \citet{amari:2016}, and \citet*{Ay|Jost|Le|Schwachhofer:2017IGbook}. Our presentation below does not use explicitly any specific parameterization of the sets $\osimplexof{\Omega}$, $\simplexof{\Omega}$, $\affineof{\Omega}$, whose topological and geometrical structure is inherited from $\R^\Omega$. The basic arguments have a ``kinetic'' flavour, in contrast with the more frequently used ``metric'' approach. I.e., I consider curves $t \mapsto p(t) \in \simplexof{\Omega}$ and look for a proper definition of velocity and acceleration.

\bigskip
The actual extension of this theory to non finite sample space requires a careful handling as most of the topological features of the finite case do not hold in the infinite case. 

One possibility is given by the so called \emph{exponential manifold}, which were first introduced in \citet{pistone|sempi:95}, and which are Banach manifolds modeled on Orlicz spaces, see the review paper \citet{pistone:2013GSI}. A different, more inclusive, option has been developed in the monograph by \citet*{ay|jost|le|schwachhofer:2017}. They use as basic topological framework the Banach space of finite signed measures with the total variation norm. The two approaches coincide when the state space is finite.

Another option would be to consider differentiable densities as the image of a geometric measure under the action of a diffeomorphism and push-forward the geometry od the group of diffeomorphism to the densities. This approach is quite interesting because the distributions are identifies with their simulation. A further possibility is the use of Kantorovich and Wasserstein geometries. \citet{montrucchio|pistone:2019-arXiv:1905.07547} discusses the finite sample space case. There is an important literature about the general case, see in particular, the monograph by \citet*{ambrosio|gigli|savare:2008}.

\subsection{The Fisher-Rao square root embedding}
\label{sec:FR}

In 1945 C.R.~Rao suggested the following construction of a Riemannian geometry on the open probability simplex $\osimplexof{\Omega}$. Nowadays, it is more commonly known under the joint name of Fisher-Rao.

Let us consider the strictly positive orthant of a sphere of radius 2 in $\R^\Omega$,
\begin{equation*}
  S_> = \setof{a \in \R^\Omega}{\normof a = 2, a(x) > 0} \ .
\end{equation*}
One has the 1-to-1 mapping of $S_>0$ to the open simplex
\begin{equation*}
\sigma \colon  a \mapsto \frac14 a^2 = \frac14 (a^2(x) \colon x \in \Omega) \ .
\end{equation*}
This mapping is a smooth mapping from the sub-manifold $S_>$ to the sub-manifold $\osimplexof{\Omega}$. In fact, the tangent at $a$ is expressed as $T_aS_> = \setof{u \in \R^\Omega}{\scalarof u a = 0}$ and the tangent space at $p = \sigma(a)$ is expressed as $T_p\osimplexof{\Omega} = \setof{U \in  \R^\Omega}{\sum_{x \in \Omega} U(x) = 0}$. In such charts, the tangent application of $\sigma$ at $a$ is the ordinary differential, $d\sigma(a)[u] = \frac12 a u$.

The same construction is frequently presented in the literature starting from the so-called embedding $\sigma^{-1} \colon p \mapsto 2 \sqrt p = a$.

Now, I want to identify the push-forward $g^\text{FR}$ with $\sigma$ of the Riemannian metric defined by $g(u,v) = \scalarof u v$ on $S_>$. For that, I require $g_p^\text{FR}(U,V) = g_a(u,v)$ if $p = \sigma(a)$, $U = d\sigma(a)[u]$, $V = d\sigma(a)[v]$, that is,
\begin{equation*}
  g_p^\text{FR}(U,V) = \sum_{x,y\in\Omega} \frac{U(x)V(x)}{p(x)} \ .
\end{equation*}

Here, I follow another approach, that leads to the same construction expressed in a different tangent bundle.

\subsection{Statistical bundle}

The main feature of this presentation of IG consists in the joint geometrical structure given to the probability simplex, that is the set of probability functions, together with the set of integrable functions. Precisely, the set of all couples $(p,f)$ where $p$ is a probability function and $f$ is a random variable is the domain of the mapping $(p,f) \mapsto \sum_x f(x) p(x)$. The two, taken together, form a \emph{vector bundle}. In the finite state case the bundle is trivial because all random variables $f$ are $p$-integrable. This is not the case when the sample space is infinite. 

This concept variously appears in the literature of IG with the name of Hilbert bundle. Cf. \citet{amari:85}, \citet{lauritzen:87ims}, \citet{murray|rice:93}, \citet{kass|vos:1997}, \citet{gibilisco|pistone:98}, \citet{amari|nagaoka:2000}, \citet{le:2017AISM}.

More precisely, let us associate to each probability function $p \in \simplexof{\Omega}$ a sub-space of the vector space of real random variables $L(p)$. In our finite setting, $L(p)$ is identified with the vector space $\R^\Omega$ if the support is full $\suppof p = \Omega$; otherwise, with $\Omega_0 = \suppof p \subset \Omega$, $L(p)$ is identified with $\R^{\Omega_0}$.

\begin{definition}\label{def:opensimplex}\ 
  \begin{enumerate}
  \item For each $p \in \osimplexof{\Omega}$, $L^2(p)$ is the vector space of real functions of $\Omega$ endowed with the inner product $\scalarat p U V = \expectat p {UV}$. It holds $L^2(p) = \R \oplus L^2_0(p)$.
\item   The \emph{statistical bundle} with base $\osimplexof{\Omega}$ is
  \begin{equation*}
    S\osimplexof{\Omega} = \setof{(p,U)}{p \in \osimplexof{\Omega}, U \in L^2_0(p)} \ .
  \end{equation*}
  \end{enumerate}
\end{definition}

\paragraph*{Remark.} Notice that $S\osimplexof{\Omega}$ is a semi-algebraic subset of the polynomial ring
  \begin{equation*}
\R[p(x),U(x) \colon x \in \Omega] \ .    
  \end{equation*}
The geometry of statistical models on a finite sample space can be studied with the tool of real algebraic geometry i.e., as \emph{Algebraic Statistics}. Cf. the monographs \citet*{pistone|riccomagno|wynn:2001}, \citet{pachter|sturmfels:2005}, \citet*{drton|sturmfels|sullivan:2009}, \citet{watanabe:2009}, \citet*{aoki|hara|takemura:2012}, \citet*{zwiernick:2016}, \citet{sullivan:2018book}. The interplay between algebraic geometry and differential geometry is discussed in the conference proceedings \citet{gibilisco|riccomagno|rogantin|wynn:2010}.

\bigskip
The geometry of the statistical bundle $S\simplexof{\Omega}$ propts for a peculiar form of velocity vectors which are defined in terms of statistical \emph{scores}, a name introduced by R.~Fisher.

Let $t \mapsto p(t) \in \simplexof{\Omega}$ be a curve which is differentiable as a curve in $\affineof{\Omega}$. Observe that $\scalarof \one {\dot p(t)} = 0$  and call $t \mapsto (p(t),\dot p(t))$ the \emph{velocity curve} which takes values in the trivial bundle $\simplexof{\Omega} \times A_0(\Omega)$, $A_0(\Omega) = \setof{v}{\one^t \cdot v = 0}$.

The following is the finite state space version of a result in \citet{ay|jost|le|schwachhofer:2017}.

\begin{proposition}
At each $t$ the support of $\dot p(t)$ is contained in the support of $p(t)$, so that there exists a curve $t \mapsto (p(t),Sp(t))$ in $S\simplexof{\Omega}$ such that 
 $\dot p(t) = Sp(t) \cdot p(t)$. The expected value of $Sp(t)$ with respect to $p(t)$ is zero.
\end{proposition}

\begin{proof}
For each $t$ and $x \in \Omega$ the condition $p(x;t) = 0$ implies that $t$ is a minimum, hence $\dot p(x;t) = 0$. It follows for all $t$ that $\dot p(t) = Sp(t) \cdot p(t)$ where $Sp(t)$ is defined by
\begin{equation}\label{eq:score}
Sp(x;t) = 
\begin{cases} 0 & \text{if $p(x;t) = 0$}, \\ \frac{\dot p(x;t)}{p(x;t)} = \derivby t \log p(x;t) & \text{if $p(x;t) > 0$.}
  \end{cases} 
\end{equation}
The expected value of $Sp(t)$ at $p(t)$ is $\sum_x Sp(x;t) \ p(x;t) = \sum_x \dot p(x;t) = 0$.\end{proof}

\begin{definition}
The (differential) \emph{score} of the differentiable curve $t \mapsto p(t) \in \simplexof{\Omega}$ is the curve in the statistical bundle $t \mapsto (p(t),Sp(t)) \in S\simplexof{\Omega}$.
\end{definition}

I first discuss the statistical geometry on the open simplex by deriving it from a \emph{vector bundle} with base $\osimplexof{\Omega}$. Later I will show that such a bundle can be identified with the tangent bundle of proper manifold structure. It is nevertheless interesting to observe that a number of geometrical properties do not require the actual definition of the statistical manifold, possibly opening the way to a new type of generalization outside the basic finite state space case.

\paragraph*{Comment.} For each $p \in \osimplexof{\Omega}$ consider the plane through the origin, orthogonal to the vector $\overrightarrow{\text{O}p}$. The set of positive probabilities each one associated to its plane forms a vector bundle which is the basic structure of our presentation of Information Geometry. Note that, because of our orientation to Statistics, we call each element of $\R^\Omega = L(\Omega)$ a \emph{random variable}.

In geometry, a mapping $F$ defined on the  probabilities $p \in \simplexof{\Omega}$ to the bundle, compatible with the bundle structure, that is
\begin{equation*}
  F \colon p \mapsto (p, F(p)) \in \simplexof{\Omega} \times \cup_{p\in\simplexof{\Omega}} S_p\simplexof{\Omega}) \ ,
\end{equation*}
such that $F(p) \in S_p\simplexof{\Omega}$---that is, $F(p)$ is a random variable and $\expectat p {F(p)} = 0$,--- is called a \emph{section}  of the vector bundle. In Statistics, such a mapping is called an \emph{estimating function} as the equation $F(\hat p,x) = 0$, $x \in \Omega$, provides an \emph{estimator}, that is a distinguished mapping from the sample space $\Omega$ to the simplex of probabilities $\simplexof{\Omega}$.
   
\paragraph*{Comment.}
The previous definition is suggested by the classical set-up of statistics, as it is revealed by the Fisher-Rao computation that leads to the notion of score. However this set-up is too narrow in a number of situation.
\begin{enumerate}
\item The probability functions in the application of interest could have zero values at some $x \in \Omega$, that is the set of interest could be the full simplex $\simplexof{\Omega}$. 
\item There are simple examples where one wants to study a neighborhood of the border of the simplex, namely something in the full affine space $\affineof{\Omega}$. See below the discussion of optimization.
\end{enumerate}

\paragraph*{Exercise}  A formal extended definition is as follows. \emph{A)} For each $\eta \in \affineof{\Omega}$ let $B_\eta$ be the vector space of random variables $U$ that are $\mu$-centered,
\begin{equation*}
  B_\eta = \setof {U \colon \Omega \to \R} {\expectat \eta U = \sum_{x\in\Omega} U(x)\ \eta(x) = 0} \ .
\end{equation*}
\emph{B)} Each $B_\eta$ is endowed with the bi-linear form
\begin{equation*}
  \scalarat \eta {U}{V} = \expectat \eta {UV} = \sum_{\setof{x\in\Omega}{\eta(x) \ne 0}} U(x)V(x) \ \eta(x) \ .
\end{equation*}
\emph{C)} The \emph{statistical bundle} of the affine space $\affineof{\Omega}$ is the linear bundle on $\affineof{\Omega}$
\begin{equation*}
  \saffineof{\Omega} = \setof{(\eta,U)}{\eta \in \affineof{\Omega}, U \in  B_\eta} \ .
\end{equation*}
\emph{D)} It is a manifold isomorphic to the open subset of the Grassmanian manifold $\operatorname{Grass}(\R^\Omega,\#\Omega-1)$ of sub-spaces $B$ that do not contain constant vectors. In fact, each fiber $B_\eta$ is a subspace of $\R^\Omega$ of co-dimension 1; Viceversa, for each subspace $B$ of dimension $(n-1)$ and not containing the constant, there is a unique complement vector $\eta$ such that $\sum_{x\in\Omega} \eta_x = 1$. 

\subsection{Natural gradient}
Let us  now discuss the notion of gradient in the statistical bundle of the open simplex.

\begin{proposition}
Let $I \ni t \mapsto p(t)$ be a $C^1$ curve in $\osimplexof{\Omega}$. For each $f \colon \Omega \to \R$,
    \begin{equation*}
      \derivby t \expectat {p(t)} {f} = \scalarat {p(t)} {f -\expectat {p(t)} f} {Sp(t)} \ ,
    \end{equation*}
    where $Sp(t) = \derivby t \logof{p(t)}$
  \end{proposition}

\begin{proof}
  \begin{align*}
    \derivby t \expectat {p(t)} {f} &= \derivby t \sum_{x \in \Omega} f(x) p(x;t) \\
 &= \sum_{x \in \Omega} f(x) \derivby t p(x;t) \\
 &= \sum_{x \in \Omega} f(x) \derivby t \log p(x;t) \ p(x;t) \\
                                    &= \expectat {p(t)} {f Sp(t)}   \quad \text{(using $\expectat {p(t)} {Sp(t)} = 0$)} \\
                                    &= \expectat {p(t)} {(f-\expectat {p(t)} f) Sp(t)} \\
    &= \scalarat {p(t)}{f - \expectat {p(t)} f}{Sp(t)} \ .
  \end{align*}
\end{proof}

Notice that $p \mapsto f - \expectat p f$ is a \emph{section} of $S\osimplexof{\Omega}$ and  $t \mapsto Sp(\cdot)$ is a \emph{lift} of $p(\cdot)$.  

\paragraph*{Example.} I have chosen not discuss here the case of the closed simplex. The condition for the existence of the differential score means that the differential score exists if and only if the curve $t \mapsto \eta(t) \in \affineof{\Omega}$ hits the faces of $\simplexof{\Omega}$ only \emph{tangentially}. For example: $n = 3$, $p(0;t) = t$, $p(1;t) = \sqrt{\frac12 - t^2}$, $p(2;t) = 1 - t - \sqrt{\frac12 - t^2}$. 

\begin{definition}[Natural gradient]
\item Given a function $f \colon \osimplexof{\Omega} \to \R$, its \emph{natural gradient} is a section
  \begin{equation*}
    \osimplexof{\Omega} \ni p \mapsto (p,\grad F(p)) \in S\osimplexof{\Omega} \ .
  \end{equation*}
 such that for each regular curve $I \ni t \mapsto p(t)$ it holds
  \begin{equation*}
    \derivby t f(p(t)) = \scalarat {p(t)} {\grad f(p(t))} {Sp(t)}, \quad t \in I \ .
  \end{equation*}
\end{definition}

\begin{proposition}[Computing $\grad$]
\item
If $f$ is a $C^1$ function on an open subset of $\R^\Omega$ containing $\osimplexof{\Omega}$, by writing $\nabla f(p) \colon \Omega \ni x \mapsto \partiald {p(x)} {f(p)}$, the following relation between the statistical gradient and the ordinary gradient holds:
\begin{equation*}
  \grad f (p) = \nabla f(p) - \expectat p {\nabla f(p)} \ .
\end{equation*}
\end{proposition}

\begin{proof}
  \begin{align*}
    \derivby t &f(p(t)) = \derivby t f(p(x;t) \colon x\in \Omega) \\ &= \sum_{x\in\Omega} \partiald {p(x)} f(p(x;t) \colon x\in \Omega) \derivby t p(x;t) \\ &=  \sum_{x\in\Omega} \partiald {p(x)} f(p(x;t) \colon x\in \Omega) \derivby t \log p(x;t) \ p(x;t) \\ &= \scalarat {p(t)} {\nabla f(p(t))} {Sp(t)} \\ &=  \scalarat {p(t)} {\nabla f(p(t))-\expectat {p(t)} {\nabla f(p(t))}} {Sp(t)} \\ &= \scalarat {p(t)} {\grad f(\eta(t))} {S\eta(t)} \ .  
  \end{align*}
\end{proof}

\paragraph*{Example: Natural gradient of the entropy}
Here is our basic example. The function
\begin{equation*}
  \entropyof p = - \sum_{x \in \Omega} p(x) \log p(x)
\end{equation*}
satisfies the conditions of the proposition with
\begin{equation*}
  \nabla \entropyof p = (x \mapsto - \log p(x) - 1) . \ 
\end{equation*}
Moreover,
\begin{equation*}
  \expectat p {\nabla \mathcal H} = \sum_{x \in \Omega} (- \log p(x) - 1) p(x) = \entropyof p - 1 \ .
\end{equation*}
It follows that
\begin{equation*}
  \grad \entropyof p = - \log p - 1 - \entropyof p + 1 = -\log p - \entropyof p \ .
\end{equation*}

The condition $\grad \entropyof q = 0$ is satisfied by a constant $\log p$.

\paragraph*{Remarks.} The Information Geometry on the simplex does not coincide with the geometry of the embedding of the simplex $\osimplexof{\Omega} \to \R^\Omega$, in the sense the statistical bundle is not the tangent bundle of these embeddingIt will become the tangent bundle of the proper geometric structure which is given by special atlases.

The vector $Sp(t) \in S_{p(t)}\osimplexof{\Omega}$ is meant to represent the relative variation of the information in a one dimensional statistical model in the sense it is a relative derivative. Geometrically, the differential score is a representation of the velocity along a curve.

Consider the level surface of $f \colon \affineof{\Omega} \to \R$ at $\eta_0 \in \affineof{\Omega}$, that is the surface $\setof{\eta \in \affineof{\Omega}}{f(\eta) = f(\eta_0)}$, and assume $\eta_0$ is not a critical point, $\grad f(\eta_0) \ne 0$. Then for each curve through $\eta_0$, $I \mapsto \eta(t)$ with $\eta(0) = \eta_0$, such that $f(\eta(t)) = f(\eta(0))$,
\begin{multline*}
  0 = \left. \derivby t f(\eta(t)) \right|_{t=0} = \\ \left. \scalarat {\eta(t)} {\grad f(p(t))} {p(t)} \right|_{t=0} = \scalarat {\eta_0} {\nabla f(\eta_0)} {S\eta(t_0)} \ ,
\end{multline*}
that is, \emph{all velocities $Sp(t_0)$ tangential to the level set are orthogonal to the statistical gradient}. Note that I have not jet defined a manifold such that the statistical bundle is equal to the tangent bundle. 

If the function $f \colon \affineof{\Omega} \to \R$ extends to a $C^1$ function on an open subset of $\R^\Omega$, then one can compute the statistical gradient via the ordinary gradient in the geometry of $\R^\Omega$, namely $\nabla f(\eta) \colon \Omega \ni x \mapsto \partiald {\eta(x)} {f(\eta)}$.  Note that the statistical gradient is zero if, and only if, the ordinary gradient is constant.

\subsection{Flows}

As already said, our emphasis is on a kinematic  approach to IG. Let us start to consider differential equations.  

\begin{definition}[Flow]\ 
\begin{enumerate}
\item Given a section $F \colon \osimplexof{\Omega}$ the \emph{trajectories along the section} are the solution of the (statistical) \emph{differential equation}
  \begin{equation*} 
    Sp(t) = F(p(t)) \ .
  \end{equation*}
\item If $F$ is defined on a open set of $\R^\Omega$ containg $\osimplexof{\Omega}$ with values in $\R^\Omega$, the statistical differential equation is equivalent to the system of ordinary differential equations
\begin{equation*} 
  \derivby t p(x;t) = p(x;t)F(x, \bm p(t)) \qquad x \in \Omega \ .
\end{equation*}
\item
The \emph{gradient flow} is the flow of the section $F = \grad f$.
\end{enumerate}
\end{definition}

The right-end-side of differential equation is 
\begin{equation*}
  G(x, \bm p) = p(x)F(x,\bm p(y)) \ ,
\end{equation*}
so that $\sum_{x \in \Omega} G(x,\bm p) = 0$. And conversely. This class of differential equations is well studied in the literature under various names, e.g., replicator equation.

\bigskip
\paragraph*{Example : Gradient flow of the expected value}
Given a random variable $f$, consider the section $F(p) = f - \expectat p f$. The flow of $F$ is the solution of
\begin{equation*}
  \dot p(x;t) = p(x;t) (f(x) - \sum_{y \in \Omega} f(y) p(y;t)) \ .
\end{equation*}

The solution is an exponential family. Consider the 1-dimensional statistical model 
\begin{equation}\label{eq:exp-model}
  p(x;t) = \expof{t f(x) - \psi(t)} p_0(x) \ ,
\end{equation}
with $\psi(t)$ normalising constant.
\begin{equation*}
  \expof {\psi(t)}  = \sum_{x \in \Omega} \expof{t f(x)} \ .
\end{equation*}
 It is a curve in $\osimplexof{\Omega}$ with $p(x;0) = p_0(x)$. The differential score is
 \begin{equation*}
   Sp(t) = \derivby t \log p(t) = f(x) - \derivby t \psi(t) \ .
 \end{equation*}
As $\expectat {p(t)} {S p(t)} = 0$, $\derivby t \psi(t) = \expectat {p(t)} f$ and we have that \Cref{{eq:exp-model}} is the solution of the natural flow starting at $p_0$.

Notice that the natural gradient of $f(p) = \expectat p f$ is precisely
   \begin{equation*}
     \grad f(p) = \nabla f(p) - \expectat p {\nabla f(p) } = F(p) \ . 
   \end{equation*}
   This is the solution of a gradient flow equation.

\bigskip
\paragraph*{Example: Gradient flow of the entropy}
Consider the equation
\begin{equation*}
Sp(t) = \grad \entropyof{p(t)} = - \log p(t) - \entropyof {p(t)} \ ,  
\end{equation*}
or
\begin{equation*}
  \derivby t \log p(t) = - \log p(t) - \entropyof {p(t)} \ .
\end{equation*}

By setting $v(x,t) = \log p(x;t)$ the equation becomes
\begin{equation*}
  \dot v(x;t) = - v(x;t) + \sum_{x \in \Omega} v(x;t) \euler^{v(x;t)} \ .
\end{equation*}

Let us look for a solution of the form
\begin{equation*}
  p(x;t) = \expof{a(t) \log p_0(x) - \psi(t)} \ .
\end{equation*}
with $a(0) = 1$, hence $p(0) = p_0$ and $\psi(0) = 0$.

In this case,
\begin{equation*}
  Sp(t) = \dot a(t) \log p_0 - \dot \psi(t) = \dot a(t) (\log p_0 - \expectat {p(t)}{\log p_0})
\end{equation*}
and
\begin{multline*}
  \entropyof{p(t)} = -\expectat {p(t)}{a(t) \log p_0 - \psi(t)} = \\ - a(t) \expectat {p(t)} {\log p_0} + \psi(t)  \ .
\end{multline*}

Plugging the previous computations into the equation,
\begin{multline*}
  \dot a(t) (\log p_0 - \expectat {p(t)}{\log p_0}) = \\ - (a(t) \log p_0 - \psi(t)) -
  (- \expectat {p(t)} {\log p_0} + \psi(t) )
\end{multline*}
which is satisfied if $\dot a(t) = -a(t)$. As $a(0)=1$,
\begin{equation*}
  p(x;t) \propto \expof{\euler^{-t} \log p_0(x)} = p_0(x)^{\euler^{-t}} \ .
\end{equation*}

In conclusion, the natural gradient flow of the entropy is an exponential family with parameter $a(t) = \euler^{-t}$, sufficient statistics $\log p_0$ and cumulant function $\psi$. 

The same exponential family as before is, in the canonical parameter,
\begin{equation*}
  p(\theta) = \expof{\theta \log p_0 - \Psi(\theta)} \propto p_0^\theta \ , \quad \theta > 0 \ .
\end{equation*}
The differential score is
\begin{equation*}
  Sp(\theta) = \log p_0 - \dot\Psi(\theta) = \log p_0 - \expectat {p(t)}{\log p_0} \ . 
\end{equation*}

\paragraph*{Example: KL-divergence}
  Consider the Kulback-Leibler divergences $p \mapsto \KL p {p_0}$ and $p \mapsto \KL {p_0} p$ and compute the respective natural gradient.

In the first case,
\begin{multline*}
    \pderivby {p(x)} \KL p {p_0} = \\ \pderivby {p(x)} \sum_{y \in \Omega} p(y) \log \frac {p(y)}{p_0(y)} = \log \frac {p(x)}{p_0(x)} + 1 \ ,
  \end{multline*}
  so that the natural gradient is
  \begin{equation*}
    \grad (p \mapsto \KL p {p_0}) = \left(p \mapsto \log \frac p {p_0} - \KL p {p_0}\right) \ .
  \end{equation*}
  The solution on the gradient flow is similar to the solution for the entropy.
  
  In the second case,
 \begin{multline*}
    \pderivby {p(x)} \KL {p_0} p = \\ \pderivby {p(x)} \sum_{y \in \Omega} p_0(y) \log \frac {p_0(y)}{p(y)} = 1 - \frac{p_0(x)}{p(x)} ,
  \end{multline*}
  so that the natural gradient is
  \begin{equation*}
    \grad (p \mapsto \KL {p_0} p ) = \left(p \mapsto 1 - \frac{p_0(x)}{p(x)} \right)\ .
  \end{equation*}

The gradient flow equation $Sp -= - \grad \KL {p_0}p$ is
\begin{equation*}
  \frac {\dot p(t)} {p(t)}= \frac{p_0}{p(t)} - 1 \quad \text{that is} \quad 
\dot p(t) = p_0 - p(t) \ ,
\end{equation*}
and the solution is
\begin{equation*}
  p(t) = p_0 + (p(0) - p_0) \euler^{-t} \ .
\end{equation*}

\paragraph*{Comment.} It is remarkable that the two variables $p$ and $q$ in $\KL p q$ are clearly associated with two different affine geometries on the statistical bundle. In fact, it is possible to derive the structure of IG from the divergence. I do not discuss this approach here and refer to the general monographs for this development.

\begin{proposition}\label{prop:gradflow}
\begin{enumerate}\ 
\item \label{item:gradflow1}
Let $f \colon \osimplexof{\Omega} \to \R$ be bounded and let $\R_+ \ni t \mapsto p(t)$ be a solution of the natural gradient flow
  \begin{equation}\label{eq:gradflow}
Sp(t) = -\grad f(p(t)) \ , \quad t > 0 \ .   
  \end{equation}
The value of $f$ along the solution, $t \mapsto f(p(t))$, is decreasing and bounded below by $\min f$. 
\item \label{item:gradflow2} Moreover, if $t \mapsto \normat {p(t)} {\grad f(p(t))} ^2 = \normat {p(t)} {Sp(t)} ^2$ is uniformly continuous, then $\lim_{t \to \infty} \normat {p(t)} {Sp(t)} = 0$. 
\item \label{item:gradflow3} Assume in addition that $p \mapsto \normat p {\grad f(p)}$ continuously extend to a the full simplex as a function $L \colon \simplexof{\Omega} \to \R$ and there exists a level set $\setof{p \in \osimplexof{\Omega}}{L(p) \le a}$ where $L$ has an unique zero $\bar p \in \simplexof{\Omega}$. In such a case, $f(p(0)) \le \alpha$ implies $\lim_{t\to\infty} p(t) = \bar p$.
\end{enumerate}
\end{proposition}

\begin{proof} \emph{\Cref{item:gradflow1}.} From \eqref{eq:gradflow} and the definition of $\grad$, 
  \begin{multline*}
    \derivby t f(p(t)) = \scalarat {p(t)} {\grad f(p(t))} {Sp(t)} = \\ - \normat {p(t)} {\grad f(p(t))} ^2 = - \normat {p(t)} {Sp(t)} ^2 \ ,
  \end{multline*}
hence
\begin{multline*}
  f(p(t)) - f(p(0)) = \\ - \int_0^t \normat {p(t)} {\grad f(p(t))} ^2 \ dt = - \int_0^t \normat {p(t)} {Sp(t)}^2 \ dt \ ,
\end{multline*}
so that $t \mapsto f(p(t))$ is decreasing and converging to a limit $\alpha \ge \min f$. 

\emph{\Cref{item:gradflow2}.} It follows from the boundedness below of $f$ that $\alpha$ is finite and moreover
\begin{multline*}
 \int_0^\infty \normat {p(t)} {\grad f(p(t))} ^2 \ dt =  \\ \int_0^\infty \normat {p(t)} {Sp(t)}^2 \ dt =  f(p_0) - \alpha \leq \max f < \infty \ .
\end{multline*}
If $t \mapsto \normat {p(t)} {\grad f(p(t))} ^2 = \normat {p(t)} {Sp(t)} ^2$ is uniformly continuous, it follows from Barbalat's lemma that $\lim_{t \to \infty} \normat {p(t)} {\grad f(p(t))} = \lim_{t \to \infty} \normat {p(t)} {Sp(t)} = 0$. 

\emph{\Cref{item:gradflow3}.} It holds $\lim_{t \to \infty} \normat {p(t)} {\grad f(p(t))} = \lim_{t \to \infty} L(p(t)) = 0$. Every solution that starts inside $\set{f \le a}$ stays in the level set. If $p \mapsto L(p)$ has a unique isolated zero at $\bar p$, then $\lim_{t\to\infty} p(t) = \bar p$.
\end{proof}

\paragraph*{Example: Expected value.} Let $f \colon \Omega \to \R$ have a unique maximum at $\bar x$ and relax $F(p) = \expectat p f$, $p \in \osimplexof{\Omega}$. It holds $\grad F(p) = f - \expectat p f$. The function $F \colon \osimplexof{\Omega}$ is bounded and $p \mapsto \normat p {\grad f}^2 = \varat p f$ is bounded and continuous on $\simplexof{\Omega}$. The trajectory is the exponential family $p(t) = \euler^{tf - \psi(t)} p_0$ and $Sp(t) = f - \psi'(t)$ is uniformly continuous because its derivative $\derivby t Sp(t) = - \psi''(t) = \varat {p(t)} f$ is bounded.

The natural gradient flow of the expected value has been intensively used as an optimization algorithm, see \citet*{malago|matteucci|pistone:2009NIPS,malago|matteucci|pistone:2011CEC,malago|matteucci|pistone:2011FOGA,malago|matteucci|pistone:2013CEC}.
\section{Connections}
\label{sec:finite-state-space-connections}

I am now going to discuss now the notion of of differentiable function on the statistical bundle which provides a sort of second order calculus. Cf.~\citet{kass|vos:1997}. 

For each random variable $U \in S_p\osimplexof{\Omega} = B_p$, it holds
\begin{equation*}
\expectat q {U - \expectat q U} = 0 \quad \text{and} \quad \expectat q{\frac pq U} = 0 \ ,
\end{equation*}
so that both $U - \expectat q U$ and $\frac pq U$ belong to $S_q\osimplexof{\Omega} = L^2_0(q)$. This prompts for the following definition. 

\begin{definition} [e- and m-transport]\ 
\label{def:transport-em}
  \begin{enumerate}
  \item The \emph{exponential transport}, or \emph{e-transport}, is the family of linear mappings defined for each $p,q \in \osimplexof{\Omega}$ by
    \begin{equation*}
      \etransport p q \colon S_p\osimplexof{\Omega} \ni U \mapsto U - \expectat q U \in S_q\osimplexof{\Omega} \ .
    \end{equation*}
  \item The \emph{mixture transport}, or \emph{m-transport}, is the family of linear mappings  for each $p,q \in \osimplexof{\Omega}$ by
    \begin{equation*}
      \mtransport p q \colon S_p\osimplexof{\Omega} \ni U \mapsto \frac p q U \in S_q\osimplexof{\Omega} \ .
    \end{equation*}
  \end{enumerate}
\end{definition}

Let us check now that the e-transport and the m-trasport are semi-groups of affine transformations which are compatible with the statistical bundle and are dual of each other with respect to the scalar product on each fiber.

\begin{theorem} \label{prop:affine-transports}The following properties hold for all $p, q, r \in \osimplexof{\Omega}$.
  \begin{enumerate}
  \item Exponential semi-group property: $\etransport q r \etransport p q = \etransport p r$.
  \item Mixture semi-group property: $\mtransport q r \mtransport p q = \mtransport p r$.
  \item Duality: $\scalarat q {\etransport p q U} {V} = \scalarat p {U} {\mtransport q p V}$, $U \in S_p \osimplexof{\Omega}$ and $V \in S_q \osimplexof{\Omega}$.
  \item \label{item:em-conservation} Conservation of the scalar product: $\scalarat q {\etransport p q U}{\mtransport p q V} = \scalarat p U V$, $U,V \in S_p\osimplexof{\Omega}$.
  \end{enumerate}
  \end{theorem}
  \begin{proof}
These are simple checks of the definitions. For example, the duality follows from
\begin{multline*}
  \scalarat q {\etransport p q U} {V} = \\ \expectat q {(U - \expectat q U)V} =  \expectat q {UV} - \expectat q U \expectat q V = \\ \expectat q {UV} = \expectat p {U\left(\frac q p V\right)} = \\ \scalarat p {U} {\mtransport q p V} \ . 
\end{multline*}

The conservation of the scalar product follows from the duality and the semi-group property:
\begin{equation*}
  \scalarat q {\etransport p q U}{\mtransport p q V} = \scalarat q {\etransport q p \etransport p q U}{V} = \scalarat p U V \ .
\end{equation*}

  \end{proof}

Each transport defines a section of the statistical bundle: given $U \in S_q\osimplexof{\Omega}$, one has the sections
\begin{equation*}
  p \mapsto \etransport q p U \quad \text{and} \quad p \mapsto \mtransport q p U \ ,
\end{equation*}
and can compute their respective flows as follows. 

\begin{proposition}\label{prop:em-flows} Let be given a random variable $U \in \ S_q\osimplexof{\Omega}$.
  \begin{enumerate}
  \item \label{item:e-flow} The flow of the section $p \mapsto \etransport q p U$, i.e., the solution of 
    \begin{equation*}
      Sp(t) = \etransport q {p(t)} U \ , \quad p(0) = p \ ,
    \end{equation*}
 is 
 \begin{equation*}
  \osimplexof{\Omega} \times \R \ni (p,t) \mapsto \euler^{t(\etransport q p U) - \psi(t)} \cdot p \ , \quad
\end{equation*}
with  $\psi(t) = \logof{\expectat p {\euler^{\etransport q p U}}}$.
\item \label{item:m-flow} The flow of the section $p \mapsto \mtransport q p U$, $U \in \ S_q\osimplexof{\Omega}$ i.e., the solution of 
    \begin{equation*}
      Sp(t) = \mtransport q {p(t)} U, \quad p(0) = p,
    \end{equation*}
 is 
 \begin{equation*}
  \osimplexof{\Omega} \times I \ni (p,t) \mapsto (1+t\mtransport q p U)p \ ,
 \end{equation*}
where $I = ]- (\max \mtransport q p U)^{-1}, - (\min \mtransport q p U)^{-1}[$.
\end{enumerate}
\end{proposition}

\begin{proof} \emph{\Cref{item:e-flow}.} This is a direct check:
    \begin{multline*}
      \derivby t (t\etransport q p U - \psi(t))  = \etransport q p U - \dot \psi(t) = \\ \etransport q p U - \expectat {p(t)} {\etransport q p U} = \etransport p {p(t)} \etransport q p U =  \etransport q {p(t)} U \ . 
    \end{multline*}

\emph{\Cref{item:m-flow}.} Assume $U \ne 0$ and let $V(x) = \mtransport q p U(x)$. As $\expectat p V = 0$, it holds both $V(x) < 0$ and $V(x) > 0$. In the first case, $1 + tV(x) > 0$ if $t \le 0$ or $t>0$ and $t < -(\min V)^{-1} \le -V(x)^{-1}$. Similarly in the other case. If $t \in I$, then $p(t) = (1+tV)p \in \osimplexof{\Omega}$ and
    \begin{multline*}
      \derivby t \logof{(1+t\mtransport q p)p} = \frac{\mtransport q p U}{1+t\mtransport q p U} = \\ \frac p {p(t)} \mtransport q p U = \mtransport p {p(t)} \mtransport q p U = \mtransport q {p(t)} U \ .
    \end{multline*}
\end{proof}

The proposition, justifies the names given to the transports.

\bigskip
Other transports are of interest. In particular, look for an \emph{isometry}
\begin{equation*}
\transport p q \colon S_p\osimplexof{\Omega} \to S_q\osimplexof{\Omega} \ , 
\end{equation*}
so that $\scalarat q {\transport p q U}{\transport p q V} = \scalarat p U V$. Compare with \cref{item:em-conservation} of \cref{prop:affine-transports}. The remaining part of this section is essentially a long exercise.

Note that $\normat q {\sqrt{\frac pq} U}^2 = \normat p U ^2$ for $U \in S_p\osimplexof{\Omega}$, but $\expectat q {\sqrt{\frac pq} U} = \expectat p {\sqrt{pq} U} = \covat p {\sqrt{pq}}{U}$ would not be zero in general. Hence, there is a linear mapping of the form
\begin{equation*}
  S_p\osimplexof{\Omega} \ni U \mapsto \sqrt{\frac pq} U + A\expectat q {\sqrt{\frac pq} U} \ .
\end{equation*}
The expected value at $q$ is 
\begin{equation*}
  \expectat q {\sqrt{\frac pq} U + A\expectat q {\sqrt{\frac pq} U}} = \left(1+\expectat q A\right)\expectat q {\sqrt{\frac pq} U} \ ,
\end{equation*}
which is zero if $\expectat q A = -1$. Under this condition, it holds $\sqrt{\frac pq} U + A\expectat q {\sqrt{\frac pq} U} \in S_q\osimplexof{\Omega}$.

\begin{widetext} Let us compute the squared norm:
\begin{multline*}
  \normat q {\sqrt{\frac pq} U + A\expectat q {\sqrt{\frac pq} U}}^2 = \normat p U ^2 + 2\expectat q {\sqrt{\frac pq} U A} \expectat q {\sqrt{\frac pq} U} + \expectat q {A^2} \expectat q {\sqrt{\frac pq} U}^2 = \\ \normat p U ^2 + \left(2\expectat q {\sqrt{\frac pq} U A}  + \expectat q {A^2} \expectat q {\sqrt{\frac pq} U} \right) \expectat q {\sqrt{\frac pq} U} = \normat p U ^2 + \expectat q {\sqrt{\frac pq} U \left(2A + \expectat q {A^2} \right) } \expectat q {\sqrt{\frac pq} U} \ .
\end{multline*}

Taking $A = - (1+\sqrt{\frac pq})(1 + \expectat q {\sqrt{\frac pq}})$ one has both $\expectat q A = -1$ and $\expectat q {\sqrt{\frac pq} U \left(2A + \expectat q {A^2} \right) } =$
\begin{multline*}
 \expectat q {\sqrt{\frac pq} U \left(-2\frac{1+\sqrt{\frac pq}}{1 + \expectat q {\sqrt{\frac pq}}} + \frac{\expectat q {\left(1+ \sqrt{\frac pq}\right)^2}}{\left(1+\expectat q {\sqrt{\frac pq}}\right)^2} \right) } = \expectat q {\sqrt{\frac pq} U \left(-2\frac{1+\sqrt{\frac pq}}{1 + \expectat q {\sqrt{\frac pq}}} + 2 \frac{1+ \expectat q {\sqrt{\frac pq}}}{\left(1+\expectat q {\sqrt{\frac pq}}\right)^2} \right) } = \\ \expectat q {\sqrt{\frac pq} U \left(-2\frac{1+\sqrt{\frac pq}}{1 + \expectat q {\sqrt{\frac pq}}} + 2 \frac{1}{1+\expectat q {\sqrt{\frac pq}}} \right) } = -2 \frac{\expectat q {\sqrt{\frac pq} U \sqrt{\frac pq} }}{1+\expectat q {\sqrt{\frac pq}}} = 0 \ .
\end{multline*}

The previous computation justifies the following definition and proposition.

\begin{definition}
The \emph{Hilbert transport}, or \emph{h-transport}, is the family of linear mappings
    \begin{equation*}
  \transport p q \colon S_p\osimplexof{\Omega} \ni U \mapsto \sqrt{\frac pq} U  -  \left(1 +  \expectat q {\sqrt{\frac pq}}\right)^{-1} \left(1 + \sqrt{\frac pq}\right) \expectat q {\sqrt{\frac pq}U} \in S_q\osimplexof{\Omega} \ ,
\end{equation*}
for all $p,q \in \osimplexof{\Omega}$.
\end{definition}

\begin{proposition} \label{prop:h-transport} The following properties hold for all $p,q \in \osimplexof{\Omega}$.
  \begin{enumerate}
  \item \label{item:h-transport-1} Inverse: $\transport q p \transport p q U = U$.
  \item \label{item:h-transport-2} Isometry: $\scalarat q {\transport p q U}{\transport p q V} = \scalarat p U V$.
  \end{enumerate}
  \end{proposition}
  \begin{proof}
 \emph{\Cref{item:h-transport-1}.} It is a long computation. Let $V = \transport p q U$, so that
     \begin{multline*}
\sqrt \frac qp V = 
\sqrt{\frac qp}\left(\sqrt{\frac pq} U  -  \left(1 +  \expectat q {\sqrt{\frac pq}}\right)^{-1} \left(1 + \sqrt{\frac pq}\right) \expectat q {\sqrt{\frac pq}U}\right) = \\ U  -  \left(1 +  \expectat q {\sqrt{\frac pq}}\right)^{-1} \left(1 + \sqrt{\frac qp}\right) \expectat q {\sqrt{\frac pq}U} \ ,
      \end{multline*}
      \begin{multline*}
 \expectat p {\sqrt \frac qp V} = \expectat p {U  -  \left(1 +  \expectat q {\sqrt{\frac pq}}\right)^{-1} \left(1 + \sqrt{\frac qp}\right) \expectat q {\sqrt{\frac pq}U}} = \\ -  \left(1 +  \expectat q {\sqrt{\frac pq}}\right)^{-1} \expectat p {1 + \sqrt{\frac qp}} \expectat q {\sqrt{\frac pq}U} = \\ -  \left(1 +  \expectat q {\sqrt{\frac pq}}\right)^{-1} \left(1 + \expectat p {\sqrt{\frac qp}}\right) \expectat q {\sqrt{\frac pq}U} \ ,
      \end{multline*}
      \begin{multline*}
\left(1+\expectat p {\sqrt\frac qp}\right)^{-1} \left(1+ \sqrt\frac qp \right) \expectat p {\sqrt \frac qp V} = \\
- \left(1+\expectat p {\sqrt\frac qp}\right)^{-1} \left(1+ \sqrt\frac qp \right) \left(1 +  \expectat q {\sqrt{\frac pq}}\right)^{-1} \left(1 + \expectat p {\sqrt{\frac qp}}\right) \expectat q {\sqrt{\frac pq}U} = \\
-  \left(1+ \sqrt\frac qp \right) \left(1 +  \expectat q {\sqrt{\frac pq}}\right)^{-1} \expectat q {\sqrt{\frac pq}U} \ ,
      \end{multline*}
and the required equality follows.
 \emph{\Cref{item:h-transport-2}.} The conservation of the norm has been already proved above. The conservation of the scalar product follows from that and from the linearity.
  \end{proof}

  Let us consider now the flow induced by the h-transport.
  
  \begin{proposition}
    Given $p \in \osimplexof{\Omega}$ and $U \in S_p\osimplexof{\Omega}$ with $\expectat p {U^2} = 1$, consider the mapping on an open interval $I$ containing 0 defined by $I \ni t \mapsto \left(\cosof{\frac t2}+\sinof{\frac t2} U\right)^2 \cdot p$, with moreover $\cosof{\frac t2}+\sinof{\frac t2} U > 0$ for all $t \in I$. Such a mapping is a regular curve in $\osimplexof{\Omega}$ such that $p(0) = p$ and $Sp(t) = \transport p {p(t)} U$.
  \end{proposition}
  \begin{proof}
    First, check that $p(t) \in \osimplexof{\Omega}$ for all $t \in I$:
    \begin{equation*}
      \expectat p {\left(\cosof{\frac t2}+\sinof{\frac t2} U\right)^2} = \cossqof{\frac t2} + 2 \cosof{\frac t2} \sinof{\frac t2} \expectat p {U} + \sinsqof{\frac t2} \expectat p {U^2} = 1 \ . 
    \end{equation*}

Second, compute the differential score:
\begin{equation*}
  Sp(t) = \derivby t \left(2 \logof{\cosof{\frac t2}+\sinof{\frac t2} U} + \log p\right)= 
2 \frac{-\frac12 \sinof{\frac t2} + \frac12 \cosof{\frac t2}U}{\cosof{\frac t2}+\sinof{\frac t2} U} = \frac{- \sinof{\frac t2} + \cosof{\frac t2}U}{\cosof{\frac t2}+\sinof{\frac t2} U} \ .
\end{equation*}

Third, compute $\transport p {p(t)} U$ in steps: $\sqrt{\frac{p}{p(t)}} = \frac1{\cosof{\frac t2}+\sinof{\frac t2} U}$;
\begin{equation*}
  \expectat {p(t)} {\sqrt{\frac{p}{p(t)}}} =  \expectat {p} {\sqrt{\frac{p(t)}{p}}} = \expectat p {\cosof{\frac t2}+\sinof{\frac t2} U} = \cosof{\frac t2} \ ; 
\end{equation*}
\begin{equation*}
  \frac{1+\sqrt{\frac{p}{p(t)}}}{1+\expectat {p(t)} {\sqrt{\frac{p}{p(t)}}}} = \left(1+\frac1{\cosof{\frac t2}+\sinof{\frac t2} U}\right) \frac1{1+\cosof{\frac t2}} = \frac{1+\cosof{\frac t2}+\sinof{\frac t2} U}{\cosof{\frac t2}+\sinof{\frac t2} U} \frac1{1+\cosof{\frac t2}} \ ;
\end{equation*}
\begin{equation*}
  \sqrt{\frac{p}{p(t)}}U = \frac{U}{\cosof{\frac t2}+\sinof{\frac t2} U} \ ;
\end{equation*}
\begin{equation*}
  \expectat {p(t)} {\sqrt{\frac{p}{p(t)}}U} =   \expectat {p} {\sqrt{\frac{p(t)}{p}}U} = \expectat p {\left(\cosof{\frac t2}+\sinof{\frac t2} U\right)U} = \sinof{\frac t2} \ ;
\end{equation*}
\begin{multline*}
  \transport p {p(t)} U = 
  \frac{U}{\cosof{\frac t2}+\sinof{\frac t2} U} - \frac{1+\cosof{\frac t2}+\sinof{\frac t2} U}{\cosof{\frac t2}+\sinof{\frac t2} U} \frac{\sinof{\frac t2}}{1+\cosof{\frac t2}} = \\ \frac{\left(1+\cosof{\frac t2}\right)U - \left(\sinof{\frac t2} + \sinof{\frac t2}\cosof{\frac t2}+\sinsqof{\frac t2} U\right)}{\left(\cosof{\frac t2}+\sinof{\frac t2} U\right)\left(1+\cosof{\frac t2}\right)} = \frac{\left(1+\cosof{\frac t2} - \sinsqof{\frac t2}\right)U - \left(\sinof{\frac t2} + \sinof{\frac t2}\cosof{\frac t2} \right)}{\left(\cosof{\frac t2}+\sinof{\frac t2} U\right)\left(1+\cosof{\frac t2}\right)} = \\ \frac{\cosof{\frac t2}U - \sinof{\frac t2}}{\cosof{\frac t2}+\sinof{\frac t2} U} \ . 
\end{multline*}
This concludes the proof.
\end{proof}
\end{widetext}

\section{Accelerations}
\label{sec:connections}

Second order geometry is usually derived from a notion of covariant derivative (connection), see e.g., \citet{lang:1995}. From the given connection one derives the relevant parallel transport. In our case, it is more natural to start from the transports already defined and derive the connections. This approach has been first applied to IG in \citet{gibilisco|pistone:98}. However, I do not construct directly the connections, but I introduce instead the accelerations associated to each transport.

Let us compute the acceleration of a curve $I \mapsto p(t)$. Let us start with the idea that the ``velocity'' is here the ``$\log$-velocity,''
\begin{equation*}
  t \mapsto (p(t),Sp(t)) = \left(p(t), \derivby t \logof {p(t)} \right) \in S\osimplexof{\Omega}
\end{equation*}

The vector $Sp(t) \in S_{p(t)}\osimplexof{\Omega}$ has to be checked against a curve in the statistical bundle, say $t \mapsto \mtransport p {p(t)} V$ for some $V \in S_p\osimplexof{\Omega}$. One can compute an acceleration as
\begin{multline*}
  \derivby t \scalarat {p(t)} {Sp(t)}{\mtransport p {p(t)} V} = \\ \derivby t \scalarat {p} {\etransport {p(t)} p Sp(t)}{V} = \scalarat {p} {\derivby t \etransport {p(t)} p Sp(t)}{V} = \\ \scalarat {p(t)} {\etransport p {p(t)} \derivby t \etransport {p(t)} p Sp(t)}{\mtransport p {p(t)} V} \ .  
\end{multline*}

\begin{definition}[e-acceleration]
  The \emph{exponential acceleration} $\eacc p(t)$ is
\begin{multline*}
\etransport p {p(t)} \derivby t \etransport {p(t)} p Sp(t) = \\  \etransport p {p(t)} \derivby t \left(\frac{\dot p(t)}{p(t)} - \expectat p {\frac{\dot p(t)}{p(t)}}\right) = \\ \etransport p {p(t)} \left(\frac{\ddot p(t)p(t)-\dot p(t)^2}{p(t)^2} - \expectat p {\frac{\ddot p(t)p(t)-\dot p(t)^2}{p(t)^2}}\right) = \\ \frac{\ddot p(t)p(t)-\dot p(t)^2}{p(t)^2} - \expectat {p(t)} {\frac{\ddot p(t)p(t)-\dot p(t)^2}{p(t)^2}} = \\ \boxed{\frac{\ddot p(t)}{p(t)} - (Sp(t))^2 + \expectat {p(t)} {(Sp(t))^2}} \ .
\end{multline*}
\end{definition}

\begin{proposition}Exponential families have null exponential acceleration.
\end{proposition}

\begin{proof}In fact, for $p(t) = \expof{tU - \psi(t)} \cdot p$, one has $\etransport {p(t)} p Sp(t) = U - \expectat p U$, so that $\derivby t \etransport {p(t)} p Sp(t) = 0$.
\end{proof}

Note that for the exponential family above,
\begin{gather*}
  (Sp(t))^2 = (u - \dot \psi(t))^2 \ , \\
  \ddot p(t) = \derivby t [p(t)(u - \dot \psi(t))] = p(t)(u-\dot\psi(t))^2 - p(t)\ddot\psi(t) \ , 
\end{gather*}
so that
\begin{multline*}
  \frac{\ddot p(t)}{p(t)} - (Sp(t))^2 + \expectat {p(t)} {(Sp(t))^2} = \\ (u-\dot\psi(t))^2 - \ddot\psi(t) - (u - \dot \psi(t))^2 + \ddot \psi(t) = 0 \ .
\end{multline*}

A second option is to compute the acceleration as
\begin{multline*}
  \derivby t \scalarat {p(t)} {Sp(t)}{\etransport p {p(t)} V} = \\ \derivby t \scalarat {p} {\mtransport {p(t)} p Sp(t)}{V} = \scalarat {p} {\derivby t \mtransport {p(t)} p Sp(t)}{V} = \\ \scalarat {p(t)} {\mtransport p {p(t)} \derivby t \mtransport {p(t)} p Sp(t)}{\etransport p {p(t)} V} \ . 
\end{multline*}

\begin{definition}[m-acceleration] The \emph{mixture acceleration} $\macc p(t)$ is 
\begin{equation*}
 \mtransport p {p(t)} \derivby t \mtransport {p(t)} p Sp(t) =  \frac {p}{p(t)} \derivby t \left(\frac {p(t)}{p} \frac {\dot p(t)}{p(t)}\right) = \boxed{\frac {\ddot p(t)}{p(t)}} \ . 
\end{equation*}
\end{definition}

\begin{proposition}Mixture models $t \mapsto (1+tU)p$ have null mixture acceleration.
\end{proposition}

\begin{proof}
  Obvious.
\end{proof}

\paragraph*{Exercise} One could define a Riemannian accelleration by
$\transport p {p(t)} \derivby t \transport {p(t)} p Sp(t)$. See some related computations in \citet{pistone:2013Entropy,pistone:2013GSI,lods|pistone:2015}.

\subsection{Taylor formula}

Let us apply the definitions of acceleration to the compute the 2nd order Taylor formula. Given a regular function $f \colon \osimplexof{\Omega} \to \R$ and a regular curve $t \mapsto p(t)$, the first derivative of $t \mapsto f(p(t))$ can be written in two ways:
\begin{multline*}
  \derivby t f(p(t)) = \scalarat {p(t)} {\grad f(p(t))}{Sp(t)} \\
\begin{aligned} &= \scalarat {p} {\mtransport {p(t)} p \grad f(p(t))}{\etransport {p(t)} p Sp(t)} \\ &=
\scalarat {p} {\etransport {p(t)} p \grad f(p(t))}{\mtransport {p(t)} p Sp(t)} 
\end{aligned}
\ . 
\end{multline*}

Using the first one,
\begin{multline*}
  \dderivby t f(p(t)) = \scalarat {p} {\derivby t \mtransport {p(t)} p \grad f(p(t))}{\etransport {p(t} p Sp(t)} + \\ 
\scalarat {p} {\mtransport {p(t)} p \grad f(p(t))}{\derivby t \etransport {p(t} p Sp(t)} = \\
\scalarat {p(t)} {\mtransport p {p(t)} \derivby t \mtransport {p(t)} p \grad f(p(t))}{Sp(t)} + \\ \scalarat {p(t)} {\grad f(p(t))}{\eacc p(t)} \ .
\end{multline*}

Assume now that $p(t) = \euler^{tU-\psi(t)} p$, $U \in S_p\osimplexof{\Omega}$, so that $\eacc p(t) = 0$ and the second term above cancels, to give
\begin{equation*}
  \dderivby t f(p(t)) = \scalarat {p(t)} {\mtransport p {p(t)} \derivby t \mtransport {p(t)} p \grad f(p(t))}{Sp(t)} \ .
\end{equation*}

\begin{definition}[m-Hessian] Define the \emph{mixture Hessian}  $\mhessianat U {f(p)}$ to be
\begin{equation*}
 \left. \mtransport p {p(t)} \derivby t \mtransport {p(t)} p \grad f(p(t)) \right|_{t=0} \in S_p\osimplexof{\Omega} 
\end{equation*}
when $p(0) = p$ and $Sp(0) = U$.
\end{definition}

The second derivative above reduces to
\begin{multline*}
    \dderivby t f(p(t)) = \scalarat {p(t)} {\mhessianat {Sp(t)} {f(p(t))}} {Sp(t)} = \\ \scalarat {p(t)} {\mhessianat {U - \dot\psi(t)} {f(p(t))}} {U - \dot\psi(t)}
\end{multline*}
and one can write for $q = \euler^{U-\psi(1)}p$, the Taylor formula
\begin{multline*}
  f(q) - f(p) = \scalarat p {\grad f(p)}{U} + \\ \int_0^1 (1-t) \scalarat {p(t)} {\mhessianat {U - \dot\psi(t)} {f(p(t))}} {U - \dot\psi(t)} \ dt = \\
 \boxed{\scalarat p {\grad f(p)}{U} + \frac12 \scalarat {p} {\mhessianat {U} {f(p)}} {U} + R_2(p,U)} \ .
\end{multline*}

\bigskip
\paragraph*{Exercise.}In a similar way one could derive a Taylor formula for the e-Hessian,
\begin{multline*}
  f(q) - f(p) =  \\
 \boxed{\scalarat p {\grad f(p)}{V} + \frac12 \scalarat {p} {\ehessianat {V} {f(p)}} {V} + R_2(p,U)} \ .
\end{multline*}
Notice that the ``increments'' $U$ and $V$ in the equations above are quite different!
The Riemannian Taylor formula could be derived along similar lines.

\section{Atlases}
\label{sec:exponentialatlas}

I have shown in the previous sections how the calculus on the statistical bundle works. I now turn to the explicit introduction of special atlases of charts on the statistical bundle. Each of atlas will have the following special properties:
\paragraph*{A.} The tangent space computed in the atlas is the corresponding fiber of the statistical bundle;
\paragraph*{B.} The atlas induces induces one of the connections.

Notice that I define an atlas which is not the maximal atlas nor the atlas deduced from the embedding into $\R^\Omega \times \R^\Omega$. It is an \emph{affine atlas}, that is all the transition maps are affine functions.

\begin{definition}[Exponential atlas]
For each $p \in \osimplexof{\Omega}$, define
\begin{multline*}
s_p \colon S\osimplexof{\Omega} \ni (q,w) \mapsto \\ \left(\log\frac q p - \expectat p {\log\frac q p }, \etransport q p w\right) \in S_p\osimplexof{\Omega} \times S_p\osimplexof{\Omega}   
\end{multline*}
\end{definition}

Recall the notation $S_p\osimplexof{\Omega} = B_p$. Notice there is a one chart for each $p \in \osimplexof{\Omega}$ and $s_p(p)=0$. One can say that $s_p$ is the chart centered at $p$..
 
\begin{proposition}[Properties of the e-atlas] \ 
  \begin{enumerate}
  \item If $u = s_p(q)$, then $q = \euler^{u-K_p(u)}\cdot p$ with $K_p(u) = \log\expectat p {\euler^u}$.
  \item The patches are
    \begin{equation*}
      s_p^{-1} \colon (u,v) \mapsto (\euler^{u-K_p(u)}\cdot p, v - dK_p(u)[v])
    \end{equation*}
  \item The transitions are given for $u,v \in B_{p_2}$ by
    \begin{multline*}
      s_{p_1} \circ s_{p_2}^{-1} \colon (u,v) \mapsto \\ \left(\etransport {p_2}{p_1} u + \log\frac{p_2}{p_1} - \expectat {p_1}{\log\frac{p_2}{p_1}}\right) \in B_{p_1} \times B_{p_1}
    \end{multline*}
  \item The tangent bundle identifies with the statistical bundle. If $p(0)=p$, then 
    \begin{equation*}
      \left. \derivby t s_p(p(t)) \right|_{t=0} = \left. \etransport {p(t)} p Sp(t) \right|_{t=0} = Sp(0) \ .
    \end{equation*}
\item The velocity computed in the chart of the lift $t \mapsto (p(t),Sp(t))$ is $t \mapsto (Sp(t),\eacc p(t))$.  
\end{enumerate}
  \end{proposition}

  \begin{proof}
    It is an exercise.
  \end{proof}

The same ideas can be applied to the m-geometry. Notice that specific properties of the finite state space case are used. If such special properties are not available, one must carefully distinguish between $B_p$ and its pre-dual $\prescript {*}{} B_p$, see \citet{pistone:2013GSI}.

\begin{definition}[Mixture atlas]
For each $p \in \osimplexof{\Omega}$, define
\begin{multline*}
\eta_p \colon S\osimplexof{\Omega} \ni (q,w) \mapsto \\ \left(\frac q p - 1, \mtransport q p w\right) \in S_p\osimplexof{\Omega} \times S_p\osimplexof{\Omega}   
\end{multline*}
\end{definition}

\begin{proposition}[Properties of the m-atlas] \ 
  \begin{enumerate}
  \item If $u = \eta_p(q)$, then $q = (1+u)p$.
  \item The patches are
    \begin{equation*}
      \eta_p^{-1} \colon (u,v) \mapsto ((1+u)p,(1+u)w) 
    \end{equation*}
  \item The transitions are
    \begin{equation*}
      \eta_{p_1} \circ \eta_{p_2}^{-1} \colon (u,v) \mapsto \left((1+u)\frac{p_2}{p_1}-1,\mtransport {p_1}{p_2} v \right)
    \end{equation*}
  \item The tangent bundle identifies with the statistical bundle. If $p(0)=p$, then 
    \begin{equation*}
      \left. \derivby t \eta_p(p(t)) \right|_{t=0} = \left. \mtransport {p(t)} p Sp(t) \right|_{t=0} = Sp(0) \ .
    \end{equation*}
\item The velocity computed in the chart of the lift $t \mapsto (p(t),Sp(t))$ is $t \mapsto (Sp(t),\macc p(t))$.  
  \end{enumerate}
\end{proposition}

\begin{proof} It is an \emph{Exercise}.\end{proof}

\subsection{Using parameters}

Even if one wants to study the geometry of the full simplex, it is possible to introduce parameters because the simplex is finite-dimensional. Parts of the presentation in this section are taken from \citet{pistone|rogantin:2015.1502.06718}.

Computations are frequently performed in a \emph{parametrization}, even in applications such as Compositional Data Analysis, which is descriptive statistics of the full simplex, see \citet{aitchinson:1986}.

A parametrization of the open simplex is a 1-to-1 mapping
\begin{equation*}
  \pi \colon \Theta \ni \bm \theta \mapsto \pi(\bm\theta) \in \osimplexof{\Omega},
\end{equation*}
$\Theta$ being an open set in $\R^n$, $n = \#\Omega - 1$. As the $j$-th coordinate curve is obtained by fixing the other $(n-1)$ components and moving $\theta_j$ only, the differential scores of each $j$-th coordinate curve are defined as the random variables
\begin{equation*}
  S_j \pi(\bm \theta) = \frac{\partial}{\partial \theta_j} \log \pi(\bm \theta), \quad j = 1,\dots,n.
\end{equation*}

The sequence $(S_j \pi(\bm \theta) \colon j = 1,\dots,n)$ is assumed to be a vector basis of the fiber  $S_{\pi(\bm\theta)}\osimplexof{\Omega}$. The expression of the inner product in such a basis is
\begin{multline*}
  \scalarat{\pi(\bm\theta)}{\sum_{i=1}^n \alpha_i S_i \pi(\bm \theta)}{\sum_{j=1}^n \beta_j S_j \pi(\bm \theta)} = \\ \sum_{i,j=1}^n \alpha_i\beta_j \scalarat{\pi(\bm\theta)}{S_i \pi(\bm \theta)}{S_j \pi(\bm \theta)} \ .
\end{multline*}
\begin{definition}[Fisher Information]
  \label{def:fisher-information}
The matrix
\begin{multline*}
  I(\bm\theta) = \left[\scalarat{\pi(\bm\theta)}{S_i \pi(\bm \theta)}{S_j \pi(\bm \theta)}\right]_{i,i=1}^n = \\ \left[\scalarat{\pi(\bm\theta)}{\partial_i \log \pi(\bm \theta)}{\partial_j \log \pi(\bm \theta)}\right]_{i,i=1}^n = \\ \left[\sum_{x \in \Omega} \frac{\partial_i \pi(x;\bm\theta) \partial_i \pi(x;\bm\theta)}{\pi(x;\bm \theta)}\right]_{i,i=1}^n
\end{multline*}
is the \emph{Fisher information matrix} of the parametrization $\pi$. Notice that the Fisher Information matrix depends on the parametrization.
\end{definition}

Consider a curve expressed in the parametrization,
\begin{equation*}
t \mapsto p(t) = \pi(\btheta(t)) \ ,  
\end{equation*}
and compute the differential score in the parametrization as
\begin{equation*}
  Sp(t) = \derivby t \log \pi(\btheta(t)) = \sum_{i=1}^n S_i\pi(\btheta(t))\dot\theta_i(t) \ .
\end{equation*}

Let us now turn to consider the expression of the natural gradient in the parametrization. Let  be given $f \colon \osimplexof{\Omega} \to \R$. The random variable $\grad f(p)$ is defined by
\begin{equation*}
  \derivby t f(p(t)) = \scalarat {p(t)} {\grad f(p(t))}{Sp(t)} \ ,
\end{equation*}
that is, for $p(t) = \pi(\btheta(t))$,
\begin{equation*}
  \derivby t f(\pi(\btheta(t))) = \scalarat {\pi(\btheta(t))} {\grad f(\pi(\btheta((t))))}{S\pi(\btheta(t))} \ .
\end{equation*}
If one writes $\tilde f(\btheta) = f\circ\pi(\btheta)$ and expresses the differential score in the basis,
\begin{multline*}
  \sum_{i=1}^n \partial_i \tilde f(\btheta(t)) \dot\theta_i(t) = \\ \sum_{i=1}^n \scalarat {\pi(\btheta(t))} {\grad f(\pi(\btheta((t))))}{S_j\pi(\btheta(t))} \dot\theta_i(t) \ .
\end{multline*}

As $\btheta$ and $\dot\btheta$ in the equation above are generic, it follows the system of equations 
\begin{equation*}
  \partial_i \tilde f(\btheta) = \scalarat {\pi(\btheta)} {\grad f(\pi(\btheta))}{S_j\pi(\btheta)} \ , \quad i=1,\dots,n \ .
\end{equation*}

This gives the form of the natural gradient that was originally proposed by \citet{amari:1998natural}.

\begin{proposition} The expression of $\grad f (\pi(\btheta))$ has components in the basis $(S_j\pi(\btheta) \colon j=1,\dots n)$ given by
\begin{equation*}
  I(\btheta)^{-1} \nabla \tilde f(\btheta) \ .
\end{equation*}
\end{proposition}

\subsection{Special parametrizations}

The \emph{common parametrization} of the (flat) simplex $\osimplexof{\Omega}$ is the projection on the \emph{solid simplex}
\begin{gather*}
\Gamma_n = \setof{\bm\eta \in \R^n}{0<\eta_j,  \sum_{j=1}^n \eta_j< 1} \ , \\
  \pi \colon \Gamma_n \ni \bm\eta \mapsto \left(1 - \sum_{j=1}^n \eta_j,\eta_1,\dots,\eta_n\right) \in \osimplexof{\Omega} \ ,
\end{gather*}
in which case $\partial_j \pi(\bm\eta)$, $j=1,\dots,n$, is the random variable with values $-1$ at $x=0$, 1 at $x=j$, 0 otherwise, hence $\partial_j \pi(\bm \eta) = \left((X=j)-(X=0)\right)$ and
\begin{equation*}
  S_j \pi(\bm \eta) = \left((X=j)-(X=0)\right)/\pi(\bm\eta).
\end{equation*}

The element $I_{jh}(\bm\eta)$ of the Fisher information matrix is
\begin{multline*}
\expectat {\pi(\bm\eta)}{\frac{(X=j)-(X=0)}{\pi(X;\bm\eta)}\frac{(X=h)-(X=0)}{\pi(X;\bm\eta)}} = \\  \sum_x \pi(x,\bm\eta)^{-1}\left((x=j)(j=h)+(x = 0)\right) = \\ \eta_j^{-1}(j=h) + \left(1 - \sum_k \eta_k\right)^{-1} \ ,
\end{multline*}
hence,
\begin{equation*}
  I(\bm\eta) = \diagof{\bm\eta}^{-1}  + \left(1 - \sum_{j=1}^n \eta_j\right)^{-1}[1]_{i,j=1}^n\ .
\end{equation*}

\paragraph*{Example}
Consider $n=3$.  The Fisher information matrix, its inverse and the determinant of the inverse are, respectively,
\begin{multline*}
 I(\eta_1,\eta_2,\eta_3)  =  (1 - \eta_1 - \eta_2-\eta_3)^{-1} \times \\
  \begin{bmatrix}
\eta_1^{-1}(1 - \eta_2-\eta_3) & 1 &1 \\ 1 & \eta_2^{-1}(1 - \eta_1-\eta_3) & 1 \\  1 &1& \eta_3^{-1}(1 - \eta_1-\eta_2)
  \end{bmatrix},
\end{multline*}
\begin{equation*}
I(\eta_1,\eta_2,\eta_3)^{-1}  =
  \begin{bmatrix}
(1 - \eta_{1}) \eta_{1} & -\eta_{1} \eta_{2} & -\eta_{1} \eta_{3}
\\
-\eta_{1} \eta_{2} & (1 - \eta_{2}) \eta_{2} & -\eta_{2} \eta_{3}
\\
-\eta_{1} \eta_{3} &  -\eta_{2} \eta_{3} & (1 - \eta_{3}) \eta_{3}
  \end{bmatrix}\ , 
\end{equation*}
\begin{equation*}
\det\left(I(\eta_1,\eta_2,\eta_3)^{-1} \right) = (1-\eta_1- \eta_2- \eta_3)\eta_1 \eta_2 \eta_3 \ .
\end{equation*}
Note that the computation of the inverse of $I(\bm\eta)$ is an application of the Sherman-Morrison formula and the computation of the determinant of $I(\bm\eta)^{-1}$ is an application of the matrix determinant lemma.

For general $n$,
\begin{proposition}\label{prop:I-1} \
\begin{enumerate}
\item The inverse of the Fisher information matrix is
\begin{equation*}
I(\bm\eta)^{-1}=\diagof{\bm\eta}  - \bm\eta \bm\eta^t \ .
\end{equation*}
\item In particular,  $I(\bm\eta)^{-1}$ is zero  on the vertexes of the simplex, only.
\item The determinant of the inverse Fisher information matrix is
\begin{equation*}
\det\left(I(\bm\eta)^{-1}\right)=\left(1- \sum_{i=1}^n  \eta_i\right)  \prod_{i=1}^n \eta_i \ .
\end{equation*}
\item The determinant of $I(\bm\eta)^{-1}$  is zero on the borders of the simplex, only.
\item On the interior of each  facet, the rank of  $I(\bm\eta)^{-1}$ is $n-1$ and the $n-1$ liner independent column vectors  generate the subspace parallel to the facet itself.
\end{enumerate}
\end{proposition}

\begin{proof} \
\paragraph*{1.} By direct computation, $I(\bm\eta) I(\bm\eta)^{-1}$ is the identity matrix.
\paragraph*{2.} The diagonal elements of  $I(\bm\eta)^{-1}$ are zero if $\eta_j=1$ or $\eta_j=0$, for $j=1,\dots,n$. If, for a given $j$,  $\eta_j=1$, then the elements of $I(\bm\eta)^{-1}$ are zero if $\eta_h=0$, $h \ne j$. The remaining case corresponds to  $\eta_j=0$  for all $j$. Then $I(\bm\eta)^{-1}=0$  on all the vertexes of the simplex.
\paragraph*{3.} It follows from Matrix Determinant Lemma.
\paragraph*{4.} The determinant factors in terms corresponding to the equations of the facets.
\paragraph*{5.} Given $i$, the conditions  $\eta_i=0$ and $\eta_j\ne 0,1$ for all $j\ne i$, define the interior of the facet orthogonal to standard base vector $e_i$. In this case the $i$-th row and the $i$-th column of $ I(\bm\eta)^{-1}$  are zero and  the complement matrix corresponds to the inverse of a Fisher information matrix in dimension $n-1$ with non zero determinant. It follows that the subspace generated by the columns has dimension $n-1$ and coincides with the space orthogonal to $\eta_i$.
Consider  the facet defined by $\left(1- \sum_{i=1}^n \eta_i\right)=0$, $\eta_i\ne 0,1$ for all $i$.  For a given $j$,  the matrix without the $j$-th row and the $j$-th column has determinant
$\left(1-\sum_{i=1,i \ne j }^n \eta_i \right)\prod_{i=1, i \ne j}^n \eta_i$. On the considered facet this determinant is different to zero and $I(\bm\eta)^{-1}$  has rank $n-1$ and their columns are orthogonal to the constant vector.
\end{proof}

\paragraph*{Exercise.} Another parametrization is the \emph{exponential parametrization} based on the exponential family with sufficient statistics $X_j = (X=j)$, $j=1,\dots,n$,
\begin{gather*}
  \pi \colon \mathbb R^n \ni \bm\theta \mapsto \expof{\sum_{j=1}^n \theta_j X_j - \psi(\bm\theta)}\frac1{n+1} \ , \quad \text{with}
\\
\psi(\bm\theta) = \logof{1 + \sum_j \euler^{\theta_j}} - \logof{n+1} \ .
\end{gather*}

Instead of using a parametrization on an open set of $\R^n$, $n = \#\Omega -1$, one could use as parameter set a manifold with the correct dimension. This approach has been made systematic in the presentation of IG by \citet{Ay|Jost|Le|Schwachhofer:2017IGbook}. The original example has been already discussed in \cref{sec:FR}.

\section{Generalised statistical bundle}
\label{sec:gener-stat-bundle}

This section is devoted to a brief discuss of a generalisation of IG based on the idea of a ``deformed'' exponential functions. That is, every positive density is represented as a random variable transformed by a function with shape similar to that of the exponential function. Such models where first introduced as a replacement of Gibbs statistics by \citet{tsallis:1988}. The presentation below follows \citet{naudts:2011GTh} and \citet*{montrucchio|pistone:2017}.

The basic relation leading to the definition of differential score can be generalised as follows. Let be given a positive real function $A$ with domain $]0,+\infty[$ and define the $A$-logarithm to be the strictly increasing function
\begin{equation*}
  \log_A(x) = \int_1^x \frac{du}{A(u)} \ .
\end{equation*}
Note that $\log_A(1) = 0$, that $\log_A$ is concave if $A$ is increasing, and that $A(x) = x$ gives the usual logarithm. The $A$-exponential is $\exp_A = \log_A^{-1}$. It holds $\exp_S'(y) = A(\exp_A(y))$.

A notable example is the Tsallis logarithm, or $q$-logarithm, which is obtained when $A(x) = x^q$ for some given real $q$. The deformed cases, i.e., $q \neq 1$, can be computed explicitly as
\begin{gather*}
  \ln_q(x) = \int_1^x \frac{du}{u^q} = \frac{x^{1-q}-1}{1-q} \\
  \exp_q(y) = \left(q+(1-q)y\right)^{\frac 1{1-q}}
  \end{gather*}

It is possible to define a $q$-differential score,
\begin{equation*}
  S^{(q)}p(t) = \derivby t \ln_q(p(t)) = \frac {\dot p(t)}{p(t)^q} \ ,
\end{equation*}
together with a $q$-statistical bundle
\begin{equation*}
  S^{(q)} \osimplexof{\Omega} = \setof{(p,U)}{p \in \osimplexof{\Omega}, \sum_{x\in\Omega} U(x) p(x)^q = 0}
\end{equation*}

It is possible to repeat in this setting the construction that led to a definition of entropy. One needs a section $U$ of the $q$-statistical bundle such that $(p,U(p)) \in S^{(q)} \in \osimplexof{\Omega}$ and $p = \exp_q((U(q) - H_q(p))$. The conditions are satisfied with
\begin{equation*}
 p = \exp_q(\ln_q p) = \exp_q(\ln_q p - \mathbb M_{p^q}(\ln_q p) + \mathbb M_{p^q}(\ln_q p)) \ ,
\end{equation*}
where $\mathbb M_{p^q}$ is the sum with weight $p^q$. The generalised entropy is
\begin{multline*}
  H_q(p) = - \mathbb M_{p^q} (\ln_q p) = - \sum_{x \in \Omega} p(x)^q \ln_q p(x) = \\ - \sum_{x \in \Omega} p(x)^q \frac{p(x)^{1-q}-1}{1-q} = \frac{-1 + \sum_{x \in \Omega} p(x)^q}{1-q} \ , 
\end{multline*}
that is, the Tsallis entropy.

Other interesting examples are the deformed logarithms defined in \citet{kaniadakis:2001PhA,kaniadakis:2001PhLA}, a special case being
\begin{equation*}
\ln_{1} = \frac{x-x^{-1}}{2} = \int_1^x \frac{2 u^2}{1+u^2} \ du \ ,   
\end{equation*}
and the deformed logarithm defined in \citet{newton:2012},
\begin{equation*}
  \ln_{\text N}(x) = \log x + x - 1 = \int_1^x \frac{u}{1+u} \ du \ .
\end{equation*}

\section{Exercises}

\paragraph*{1.}
  Study the curve 
  \begin{equation*}
    t \mapsto \left(\frac12 +\frac12(1-tU)^2-\frac t2 \expectat p {U^2}\right) \cdot p , \quad U \in S_p\osimplexof{\Omega} \ .
  \end{equation*}
See \citet{eguchi:2005igaia}.

\paragraph*{2..}
 Study the curve 
    \begin{equation*}
      p(t) = \left(tU + \sqrt{1 - t^2 \expectat p {U^2}}\right)^2 \cdot p \ .
    \end{equation*}
    See \citet{burdet|combe|nencka:2001}.

\paragraph*{3.}
 Check whether
 \begin{equation*}
   t \mapsto p(t) = (1+t^2\expectat p {U^2})^{-1} (1+tU)^2 \cdot p
 \end{equation*}
 is the flow of the h-transport.

\paragraph*{4.}
Chose a $U \in S_p\osimplexof{\Omega}$or each with unit $p$-norm, $\expectat p {U^2} = 1$, and consider the model  
\begin{equation*}
   t \mapsto p(t) = \frac{(1+(\sinh t)U)^2}{\cosh^2 t} \cdot p \ ,
 \end{equation*}
where $1+(\sinh t)U > 0$ if $t \in I$, $I$ neighborhood of 0, so that $p(t) > 0$ and $\expectat p {(1+(\sinh t)U)^2} = 1 + \sinh^2 t = \cosh^2 t$. $I \ni t \mapsto p(t)$ is a regular curve with differential score
\begin{equation*}
  Sp(t) = 2 \left(\frac{(\cosh t)U}{1+(\sinh t)U} - \frac{\sinh t}{\cosh t}\right) \ .
\end{equation*}

Compute $\transport p {p(t)} U$.

\paragraph*{5} It is possible to define the Riemannian atlas through the embedding of $\osimplexof{\Omega}$ onto the the tangent space of the unit sphere, see a presentation in the style of the previous ones in \citet{pistone:2013GSI}.

\end{document}